\documentclass[10pt,reqno]{article}

\usepackage{etex}
\usepackage{enumitem}
\usepackage{setspace}
\usepackage{amsmath}

\usepackage{tabu}
\usepackage{tikz}
\usepackage[margin=.8in]{geometry}
\usepackage[matrix,arrow,curve,frame]{xy}
\usepackage{hyperref}
\usepackage{amsthm}
\usepackage{textcomp}
\usepackage{amsfonts}
\usepackage{amssymb}
\usepackage{wasysym}
\usepackage{mathrsfs}
\usepackage{mathtools}

\definecolor{my-linkcolor}{rgb}{0.75,0,0}
\definecolor{my-citecolor}{rgb}{0.1,0.57,0}
\definecolor{my-urlcolor}{rgb}{0,0,0.75}
\hypersetup{
	colorlinks, 
	linkcolor={my-linkcolor},
	citecolor={my-citecolor}, 
	urlcolor={my-urlcolor}
}

\title{On the tensor structure of modules for compact orbifold vertex operator algebras}
\author{Robert McRae\\
\small \it 
Department of Mathematics, Vanderbilt University\\
\small \it 1326 Stevenson Center, Nashville, TN 37240\\
\small \textit{E--mail address:} \texttt{robert.h.mcrae@vanderbilt.edu}} \date{}

\numberwithin{equation}{section}

\theoremstyle{definition}\newtheorem{rema}{Remark}[section]
\theoremstyle{plain}\newtheorem{propo}[rema]{Proposition}
\newtheorem{theo}[rema]{Theorem}
\theoremstyle{definition}\newtheorem{defi}[rema]{Definition}
\theoremstyle{plain}\newtheorem{lemma}[rema]{Lemma}
\newtheorem{corol}[rema]{Corollary}
\theoremstyle{definition}\newtheorem{exam}[rema]{Example}
\theoremstyle{definition}\newtheorem{assum}[rema]{Assumption}

\newcommand{\cY}{\mathcal{Y}}
\newcommand{\cA}{\mathcal{A}}
\newcommand{\cR}{\mathcal{R}}
\newcommand{\cC}{\mathcal{C}}
\newcommand{\CC}{\mathbb{C}}
\newcommand{\ZZ}{\mathbb{Z}}
\newcommand{\NN}{\mathbb{N}}
\newcommand{\RR}{\mathbb{R}}
\newcommand{\tens}{\boxtimes}
\newcommand{\vac}{\mathbf{1}}
 \DeclareMathOperator{\rep}{Rep}
 \DeclareMathOperator{\Endo}{End}
\newcommand{\even}{{\bar{0}}}
\newcommand{\odd}{{\bar{1}}}

\begin{document}

\bibliographystyle{alpha}

\maketitle

\abstract{
\noindent Suppose $V^G$ is the fixed-point vertex operator subalgebra of a compact group $G$ acting on a simple abelian intertwining algebra $V$. We show that if all irreducible $V^G$-modules contained in $V$ live in some braided tensor category of $V^G$-modules, then they generate a tensor subcategory equivalent to the category $\rep G$ of finite-dimensional representations of $G$, with associativity and braiding isomorphisms modified by the abelian $3$-cocycle defining the abelian intertwining algebra structure on $V$. Additionally, we show that if the fusion rules for the irreducible $V^G$-modules contained in $V$ agree with the dimensions of spaces of intertwiners among $G$-modules, then the irreducibles contained in $V$ already generate a braided tensor category of $V^G$-modules. These results do not require rigidity on any tensor category of $V^G$-modules and thus apply to many examples where braided tensor category structure is known to exist but rigidity is not known; for example they apply when $V^G$ is $C_2$-cofinite but not necessarily rational. When $V^G$ is both $C_2$-cofinite and rational and $V$ is a vertex operator algebra, we use the equivalence between $\mathrm{Rep}\,G$ and the corresponding subcategory of $V^G$-modules to show that $V$ is also rational. As another application, we show that a certain category of modules for the Virasoro algebra at central charge $1$ admits a braided tensor category structure equivalent to $\rep SU(2)$, up to modification by an abelian $3$-cocycle.

\bigskip


\bigskip

}

\tableofcontents

\section{Introduction}

In \cite{DLM}, Dong, Li, and Mason proved the following Schur-Weyl-duality-type result for a simple vertex operator algebra $V$: if $G$ is a compact Lie group of automorphisms acting continuously on $V$, then $V$ is semisimple as a module for $G\times V^G$, where $V^G$ is the vertex operator subalgebra of $G$-fixed points. Specifically, $V=\bigoplus_{\chi\in\widehat{G}} M_\chi\otimes V_\chi$, where the sum runs over all finite-dimensional irreducible characters of $G$, $M_\chi$ is the finite-dimensional irreducible (continuous) $G$-module corresponding to $\chi$, and the $V_\chi$ are non-zero, distinct, irreducible $V^G$-modules. This decomposition thus sets up a correspondence between the category $\rep G$ of finite-dimensional continuous $G$-modules and the semisimple subcategory $\cC_V$ of $V^G$-modules generated by the $V_\chi$. It is then natural to ask whether this correspondence preserves additional structure on the categories under consideration, in particular tensor category structure. 

In \cite{KirillovOrbifoldI}, Kirillov constructed a tensor equivalence $\Phi: \rep G\rightarrow \cC_V$ in a general categorical setting, but under rather strong conditions: $G$ is finite, and the modules $V_\chi$ are contained in a rigid semisimple tensor category of $V^G$-modules. While the category $\cC_V$ is certainly semisimple, it is unfortunately difficult to establish rigidity for a tensor category of vertex operator algebra modules. Up to this point, the only general theorem establishing rigidity for vertex operator algebra module categories by Huang \cite{H-rigidity} applies to so-called ``strongly rational'' vertex operator algebras whose full module categories satisfy strong semisimplicity assumptions. On the other hand, many vertex operator algebras, including $C_2$-cofinite but non-rational ones \cite{H-cofin}, admit non-semisimple tensor categories of modules, and it is not known in general when these tensor categories are rigid. Thus we would like to show that the functor $\Phi$ from \cite{KirillovOrbifoldI} is a tensor equivalence \textit{without} using rigidity on any category of $V^G$-modules.

For the case of $G$ abelian, \cite{Miy}, \cite{CarM}, and \cite{CKLR} established rigidity for modules in $\cC_V$, thus effectively showing that $\Phi$ is an equivalence without the \textit{a priori} assumption of rigidity. Here, we prove that $\Phi$ is an equivalence of symmetric tensor categories for general (non-abelian) compact $G$ under essentially maximally general conditions: we only require $V^G$ to actually \textit{have} a braided tensor category of modules that includes the $V_\chi$. Actually, we prove a slightly more general result: the simple vertex operator algebra $V$ can be replaced by a simple abelian intertwining algebra. Abelian intertwining algebras, introduced in \cite{DL}, are a kind of generalized vertex operator algebra graded by an abelian group $A$, where the usual associativity and commutativity properties of the vertex operator are modified by an abelian $3$-cocycle on $A$ with values in $\CC^\times$. We verify that the Schur-Weyl duality result of \cite{DLM} still holds in the abelian intertwining algebra setting, and then show that if $V^G$ has a tensor category of modules including the $V_\chi$, then $\cC_V$ is braided equivalent to a modification of $\rep G$ involving the $3$-cocycle on $A$.

As the proof of our main result does not use rigidity for any category of $V^G$-modules, $\cC_V$ inherits rigidity from $\rep G$ as a corollary. Specifically, the tensor functor $\Phi$ sends duals of $G$-modules to duals of $V^G$-modules. Moreover, since tensor functors are compatible with all tensor category structures, in particular with unit and associativity isomorphisms, the rigidity axioms for duals of $V^G$-modules hold because they do for duals of $G$-modules. We expect this corollary to have applications in the study of larger categories of $V^G$-modules as well as twisted $V$-modules. In this paper, we provide one application: when $V$ is a simple CFT-type vertex operator algebra and $G$ is finite, we show that if $V^G$ is strongly rational, then so is $V$. This result has previously appeared in \cite{ADJR}, based on \cite[Theorem 3.5]{HKL}, but under the strong additional assumption that all irreducible $V^G$-modules other than $V^G$ itself have strictly positive conformal weight gradings.

It turns out that the assumption of our main theorem, that $V^G$ admits a suitable braided tensor category of modules, is rather deep. The construction of braided tensor category structure on modules for a vertex operator algebra is based on the (logarithmic) vertex tensor category theory of Huang-Lepowsky-(Zhang) (\cite{HL-VTC}-\cite{HL-tensor3}, \cite{H-tensor4}, \cite{HLZ1}-\cite{HLZ8}; see also the review article \cite{HL-rev}). The construction of associativity isomorphisms is particularly difficult, related to convergence and expansion properties of conformal-field-theoretic $4$-point functions. Nevertheless, many vertex algebraic tensor categories have been constructed, including for example the full categories of (grading-restricted, generalized/logarithmic) modules for $C_2$-cofinite vertex operator algebras (\cite{H-cofin}). So the results of this paper apply when $V^G$ is $C_2$-cofinite; this is known when $V$ is $C_2$-cofinite and $G$ is finite solvable (\cite{Miy2}), and one conjecures that the same should hold for any finite group $G$. 

In this paper, we also prove a result on existence of braided tensor category structure on $\mathcal{C}_V$: if the spaces of intertwining operators among the $V^G$-modules $V_\chi$ are isomorphic to the spaces of intertwiners among the $G$-modules $M_\chi$, then $\cC_V$ admits braided tensor category structure. In this situation, all intertwining operators among the $V^G$-modules $V_\chi$ are derived from the vertex operator for the larger algebra $V$, which has the associativity properties needed to obtain associativity isomorphisms in $\cC_V$. This result applies for example to the central charge $c=1$ Virasoro vertex operator algebra $L(1,0)$, which is the $SU(2)$-fixed point subalgebra of the $\mathfrak{sl}_2$-weight lattice abelian intertwining algebra; see Section \ref{sec:exams} for more details.

The main results of this paper will have further applications to appear in upcoming work. In \cite{McR}, we study any (not necessarily semisimple or rigid) braided tensor category $\cC$ of $V^G$-modules containing $\cC_V$. When $G$ is finite, $V$ is an algebra object in $\cC$ (\cite{HKL}), and we use the rigidity of $\cC_V$ to show that every module in $\cC$ having an associative, unital action of $V$ is a direct sum of $g$-twisted $V$-modules for $g\in G$. This result has two applications:
\begin{itemize}
 \item First, perhaps the most prominent question regarding group actions on vertex operator algebras is the converse of the first application discussed above: whether $V^G$ is strongly rational when $V$ is strongly rational and $G$ is finite. This was recently proved for $G$ finite cyclic (and by extension, $G$ finite solvable) in \cite{CarM}. In \cite{McR}, we prove that if $V$ is strongly rational and $\cC$ is a category of $V^G$-modules with braided tensor category structure, then $\mathcal{C}$ is semisimple. In particular, if we know $V^G$ is also $C_2$-cofinite, we may take $\cC$ to be the full $V^G$-module category, and it follows that $V^G$ is rational. The main idea here is that semisimplicity of $\cC$ follows from semisimplicity of the category of $V^G$-modules in $\cC$ admitting an associative, unital $V$-action, and semisimplicity of this category follows because any indecomposable object is a $g$-twisted $V$-module for some individual $g\in G$ and because each $V^{\langle g\rangle}$ is rational by \cite{CarM}.
 
 \item Secondly, we show in \cite{McR} that any braided tensor category $\cC$ of $V^G$-modules, with $G$ finite, is braided equivalent to the $G$-equivariantization of the $G$-crossed category of twisted $V$-modules in $\cC$. This result partially resolves a conjecture of Runkel that the braided tensor category constructed in \cite{Ru} is equivalent to the (non-semisimple) braided tensor category of modules for the even vertex operator subalgebra of the symplectic fermion vertex operator superalgebra $SF(d)$. In fact, the category in \cite{Ru} seems to be the $\ZZ/2\ZZ$-equivariantization of the category of untwisted and parity-twisted $SF(d)$-modules; thus full resolution of Runkel's conjecture amounts to checking this carefully.
\end{itemize}

The methods used in this paper are a synthesis of vertex algebraic and tensor categorical techniques. The main difficulty in showing that the functor $\Phi: \rep G\rightarrow \cC_V$ is a braided tensor equivalence is showing that a natural transformation
\begin{equation*}
 J_{M_1,M_2}: \Phi(M_1)\tens\Phi(M_2)\rightarrow\Phi(M_1\otimes M_2),
\end{equation*}
where $M_1$, $M_2$ are objects of $\rep G$, constructed by Kirillov in the case of $G$ finite, is actually an isomorphism. In \cite{KirillovOrbifoldI}, the proof of both the injectivity and surjectivity of $J_{M_1,M_2}$ heavily used the rigidity of objects in $\cC_V$. In our vertex algebraic context, we show that surjectivity is a fairly easy consequence of a lemma of Dong and Mason from \cite{DM1}. The proof of injectivity is much more involved and can perhaps be viewed as a translation of the argument in \cite{Miy}, \cite{CarM}, and \cite{CKLR} for $G$ abelian into tensor-categorical language, but considerably generalized to the non-abelian setting where, for example, evaluations and coevaluations in $\rep\,G$ are not isomorphisms. Our proof of injectivity illustrates the value of using tensor-categorical methods to study vertex operator algebra module categories: proving the result directly using intertwining operators would require careful use of complex analysis to deal with compositions of up to three intertwining operators. Fortunately, the necessary complex analysis has been done in \cite{HLZ5}-\cite{HLZ8}, so we can freely use the resulting associativity isomorphisms and pentagon identity, among other tensor-categorical structures.

We now overview the remaining contents of this paper. In Section \ref{subsec:vrtxtenscat}, we review terminology and structures from the vertex tensor category theory of Huang-Lepowsky-(Zhang) that will be used later in the work; we also prove that a vertex tensor functor (in the sense of \cite[Section 3.6]{CKM}) between vertex tensor categories induces a braided tensor functor between corresponding braided tensor categories. In Section \ref{sec:gpmodcats}, we explain how to modify the ribbon category of finite-dimensional representations of a compact group $G$ by an abelian $3$-cocycle on a discrete abelian group $A$, where $G$ contains the compact dual group $\widehat{A}$. In Section \ref{sec:Abintwalg}, we recall from \cite{DL} the notion of abelian intertwining algebra associated to an abelian $3$-cocycle on an abelian group $A$, as well as some basic properties of abelian intertwining algebras and their automorphisms.

In Section \ref{sec:SW-duality}, we show that the Schur-Weyl duality result of \cite{DLM} still holds for a simple abelian intertwining algebra $V$ with compact Lie group $G$ of automorphisms, and we use this result in Section \ref{sec:Phi} to show that the functor $\Phi$ from $\mathrm{Rep}\,G$ to $V^G$-modules constructed in \cite{KirillovOrbifoldI} is fully faithful and thus an equivalence of categories onto its image. We also show that $\Phi$ induces an injection from $G$-module intertwiners to $V^G$-module intertwining operators.

We prove the main theorems of this paper in Section \ref{sec:maintheorems}. In Section \ref{subsec:firstmaintheo}, we show that if the injection on intertwiner spaces from Section \ref{sec:Phi} is also surjective, then the semisimple category $\cC_V$ of $V^G$-modules whose simple objects occur in $V$ admits vertex tensor category structure in the sense of Huang-Lepowsky-(Zhang) and therefore also braided tensor category structure. On the other hand, we show in Section \ref{subsec:secondmaintheo} that if $\cC_V$ or some larger category of $V^G$-modules admits vertex and braided tensor category structure, then $\mathcal{C}_V$ is braided tensor equivalent to an abelian $3$-cocycle modification of $\mathrm{Rep}\,G$, so that in particular the induced map on intertwiners is an isomorphism. We discuss examples and applications in Section \ref{sec:exams}. Especially, we construct a new braided tensor category structure on a semisimple category of modules for the simple Virasoro vertex operator algebra $L(1,0)$, and we show that if $V$ is a simple vertex operator algebra and $V^G$ is strongly rational, then $V$ is also strongly rational, with no assumption on the gradings of $V^G$-modules.

\paragraph{Acknowledgements}
I would like to thank Thomas Creutzig, Shashank Kanade, Florencia Orosz Hunziker, and Jinwei Yang for comments and discussions, as well as the referee for suggestions and corrections. Theorem \ref{VregifVGis} and its proof were inspired by the question https://mathoverflow.net/questions/321416/
asked by Bin Gui on MathOverflow. This work is supported in part by National Science Foundation grant DMS-1362138.

\section{Preliminaries}

\subsection{Vertex tensor categories}\label{subsec:vrtxtenscat}

The notion of \textit{vertex tensor category} was introduced by Huang and Lepowsky in \cite{HL-VTC} to describe the tensor structures exhibited by suitable module categories for a vertex operator algebra, originally constructed in \cite{HL-tensor1}-\cite{HL-tensor3}, \cite{H-tensor4}. A key feature of vertex tensor categories is the existence of a tensor product for each punctured sphere with two ingoing and one outgoing punctures, equipped with local coordinates at each puncture. Such tensor products are motivated by the connection, made precise by Huang in \cite{H-book}, between vertex operator algebras and the geometric picture of conformal field theory introduced by Vafa in \cite{V} (which in turn is a variant of the similar picture developed by Segal in \cite{S}).

In the connection between vertex operator algebras and Vafa conformal field theory, the vertex operator $Y(\cdot, z)\cdot$ for a vertex operator algebra, with formal variable $x$ replaced by $z\in\CC^\times$, corresponds to the sphere $P(z)$ with ingoing punctures at $0$ and $z$ and one outgoing puncture at $\infty$, with standard local coordinates at the punctures. Specifically, the local coordinates for $P(z)$ are $w\mapsto w$, $w\mapsto w-z$, and $w\mapsto 1/w$ around $0$, $z$, and $\infty$, respectively. It turns out (see \cite{H-book}) that vertex operators corresponding to all other choices of local coordinates at the punctures can be obtained from $Y(\cdot, z)\cdot$ using the Virasoro action on the vertex operator algebra. For this reason, we will somewhat abuse terminology and use the term ``vertex tensor category structure'' to refer simply to the tensor structures pertaining to the sphere $P(z)$, which we now describe following \cite{HLZ1}-\cite{HLZ8}.

Let $\mathcal{C}$ be a category of grading-restricted generalized modules (in the sense of \cite[Definitions 2.12 and 2.25]{HLZ1}) for a vertex operator algebra $V$. In particular, a module $W$ in $\cC$ has a conformal weight grading $W=\bigoplus_{h\in\CC} W_{[h]}$ where for any $h\in\CC$, $W_{[h]}$ is the (finite-dimensional) generalized $h$-eigenspace for the Virasoro operator $L(0)$ on $W$. The \textit{graded dual} $W'=\bigoplus_{h\in\CC^\times} W_{[h]}^*$ is also a grading-restricted generalized $V$-module. We use $\overline{W}$ to denote the \textit{algebraic completion} of $W$, which is the direct product (as opposed to direct sum) of the conformal weight spaces $W_{[h]}$. A useful characterization of the algebraic completion is that it is the full vector space dual of $W'$: $\overline{W}=(W')^*$.

For $h\in\CC$, we use $\pi_h$ to denote the projection from $\overline{W}$ to $W_{[h]}$. We now recall the notion of $P(z)$-intertwining map from \cite{HLZ3}:
\begin{defi}
 Let $W_1$, $W_2$, and $W_3$ be $V$-modules in $\cC$. A \textit{$P(z)$-intertwining map} of type $\binom{W_3}{W_1\,W_2}$ is a linear map
 \begin{equation*}
  I: W_1\otimes W_2\rightarrow\overline{W_3}
 \end{equation*}
satisfying the following conditions:
\begin{enumerate}
 \item \textit{Lower truncation}: For any $h\in\CC$, $w_1\in W_1$, and $w_2\in W_2$, $\pi_{h-n}I(w_1\otimes w_2) = 0$ for $n\in\NN$ sufficiently large.
 
 \item The \textit{Jacobi identity}:
 \begin{align*}
  x_0^{-1}\delta\left(\frac{x_1-z}{x_0}\right) Y_{W_3}(v, x_1) I(w_1\otimes w_2) & - x_0^{-1}\delta\left(\frac{-z+x_1}{x_0}\right) I(w_1\otimes Y_{W_2}(v, x_1)w_2)\nonumber\\
  &= z^{-1}\delta\left(\frac{x_1-x_0}{z}\right) I(Y_{W_1}(v,x_0)w_1\otimes w_2)
 \end{align*}
for all $v\in V$, $w_1\in W_1$, and $w_2\in W_2$.
\end{enumerate}

\end{defi}

From the definition, it is easy to see that a $P(z)$-intertwining map can be obtained from a logarithmic intertwining operator
\begin{equation*}
 \cY(\cdot, x)\cdot: W_1\otimes W_2\rightarrow W_3[\log x]\lbrace x\rbrace
\end{equation*}
by specializing the formal variable $x$ to $z\in\CC^\times$ using a choice of branch of logarithm. For simplicity, we will always use the branch 
\begin{equation*}
 \log z = \log \vert z\vert +i\arg z
\end{equation*}
for which $0\leq\arg z<2\pi$. Thus given an intertwining operator $\mathcal{Y}$, we will use $I_\mathcal{Y}$ to denote the $P(z)$-intertwining map
\begin{equation*}
 I_\mathcal{Y} =\cY(\cdot, e^{\log z})\cdot.
\end{equation*}
Conversely, \cite[Proposition 4.8]{HLZ3} shows that any intertwining operator of type $\binom{W_3}{W_1\,W_2}$ comes from a $P(z)$-intertwining map; specifically, the inverse to the map $\cY\mapsto I_{\mathcal{Y}}$ is the map $I\mapsto\cY_I$ given by
\begin{equation}\label{YfromI}
 \cY_I(w_1, x)w_2=\left(\frac{e^{\log z}}{x}\right)^{-L(0)} I\left(\left(\frac{e^{\log z}}{x}\right)^{L(0)} w_1\otimes \left(\frac{e^{\log z}}{x}\right)^{L(0)} w_2\right)
\end{equation}
for $w_1\in W_1$, $w_2\in W_2$.

Now we recall the notion of $P(z)$-tensor product in $\mathcal{C}$ from \cite{HLZ3}:
\begin{defi}
 Suppose $W_1$ and $W_2$ are $V$-modules in $\mathcal{C}$. A \textit{$P(z)$-tensor product} of $W_1$ and $W_2$ in $\cC$ is a pair $(W_1\tens_{P(z)} W_2, \boxtimes_{P(z)})$, where $W_1\tens_{P(z)} W_2$ is an object of $\cC$ and $\boxtimes_{P(z)}$ is a $P(z)$-intertwining map of type $\binom{W_1\tens_{P(z)} W_2}{W_1\,W_2}$, such that the following universal property holds:
 
 For any $V$-module $W_3$ in $\cC$ and $P(z)$-intertwining map $I$ of type $\binom{W_3}{W_1\,W_2}$, there is a unique morphism $f_I: W_1\tens_{P(z)} W_2\rightarrow W_3$ in $\cC$ such that $$I=\overline{f_I}\circ\tens_{P(z)},$$
 where $\overline{f_I}$ is the canonical extension of $f_I$ to $\overline{W_1\tens_{P(z)} W_2}$.
\end{defi}

\begin{rema}\label{tensprodsindiffC}
 The identity of the $P(z)$-tensor product of $W_1$ and $W_2$ may possibly depend on the category $\cC$ under consideration, given that it is defined in terms of intertwining maps only from the category $\cC$.
\end{rema}

\begin{rema}
 We use the notation $w_1\boxtimes_{P(z)} w_2=\boxtimes_{P(z)}(w_1\otimes w_2)$ for $w_1\in W_1$ and $w_2\in W_2$, in analogy with the notation for a tensor product bilinear map in the category of vector spaces. By \cite[Proposition 4.23]{HLZ3}, the $P(z)$-tensor product module is spanned by vectors $\pi_h(w_1\tens_{P(z)} w_2)$ for $h\in\CC$, $w_1\in W_1$, and $w_2\in W_2$.
\end{rema}

\begin{rema}
 Equation \eqref{YfromI} shows how to obtain a tensor product intertwining operator $\cY_{\boxtimes_{P(z)}}$ from the tensor product $P(z)$-intertwining map. We will frequently need to evaluate such tensor product intertwining operators at different values of $z$, so here we record the following consequence of \eqref{YfromI}:
 \begin{align*}
  \cY_{\boxtimes_{P(z_2)}}(w_1, e^{\ell(z_1)})w_2 & = e^{(\ell(z_1)-\log z_2)L(0)} \boxtimes_{P(z_2)}\left(e^{-(\ell(z_1)-\log z_2)L(0)}w_1\otimes e^{-(\ell(z_1)-\log z_2)L(0)}w_2\right)\nonumber\\
  & = e^{(\ell(z_1)-\log z_2)L(0)}\left(e^{-(\ell(z_1)-\log z_2)L(0)}w_1\tens_{P(z_2)} e^{-(\ell(z_1)-\log z_2)L(0)}w_2\right),
 \end{align*}
where $z_1,z_2\in\CC^\times$ and $\ell(z_1)$ is any branch of logarithm of $z_1$ (in fact, sometimes we have $z_1=z_2$ and $\ell$ a possibly different branch from $\log$). Note that here, $\cY_{\boxtimes_{P(z_2)}}(\cdot, e^{\ell(z_1)})\cdot$ is a $P(z_1)$-intertwining map of type $\binom{W_1\tens_{P(z_2)} W_2}{W_1\,W_2}$.
\end{rema}

Assuming that $P(z)$-tensor products exist in the category $\cC$, they define functors $\tens_{P(z)}: \cC\times\cC\rightarrow\cC$. The $P(z)$-tensor product of morphisms $f_1: W_1\rightarrow X_1$ and $f_2: W_2\rightarrow X_2$ is defined to be the unique homomorphism (guaranteed by the universal property of the $P(z)$-tensor product) such that the diagram
\begin{equation*}
 \xymatrixcolsep{4pc}
 \xymatrix{
 W_1\otimes W_2 \ar[d]^{\tens_{P(z)}} \ar[r]^{f_1\otimes f_2} & X_1\otimes X_2 \ar[d]^{\tens_{P(z)}} \\
 \overline{W_1\tens_{P(z)} W_2} \ar[r]^{\overline{f_1\tens_{P(z)} f_2}} & \overline{X_1\tens_{P(z)} X_2} \\
 }
\end{equation*}
commutes. By \cite[Proposition 4.23]{HLZ3}, $f_1\tens_{P(z)} f_2$ is thus completely characterized by the relation
\begin{equation*}
 \overline{f_1\tens_{P(z)} f_2}(w_1\tens_{P(z)} w_2)=f_1(w_1)\tens_{P(z)} f_2(w_2)
\end{equation*}
for $w_1\in W_1$, $w_2\in W_2$.

Now we will say that $\cC$ admits vertex tensor category structure if $P(z)$-tensor products exist in $\cC$ and we have the following natural isomorphisms, described mostly in \cite{HLZ8} (see also the exposition in \cite[Section 3.3]{CKM}):
\begin{itemize}
 \item \textit{Parallel transport isomorphisms}: For any continuous path $\gamma$ in $\CC^\times$ beginning at $z_1$ and ending at $z_2$, there is a natural isomorphism $T_\gamma: \tens_{P(z_1)}\rightarrow\tens_{P(z_2)}$ characterized by
 \begin{equation*}
  \overline{T_{\gamma; W_1, W_2}}(w_1\tens_{P(z_1)} w_2) =\cY_{\boxtimes_{P(z_2)}}(w_1, e^{\ell(z_1)})w_2
 \end{equation*}
for all $w_1\in W_1$, $w_2\in W_2$, where $\ell(z_1)$ is the branch of logarithm determined by $\log z_2$ and the path $\gamma$. If $\gamma$ is a path contained in $\CC^\times$ with a branch cut along the positive real axis (that is, $\gamma$ never crosses the positive real axis or approaches it from the lower half plane), then $\ell(z_1)$ is simply $\log z_1$, and we use $T_{z_1\to z_2}$ to denote the corresponding parallel transport isomorphism.

\item \textit{$P(z)$-unit isomorphisms}: For any $z\in\CC^\times$, there are natural isomorphisms $l_{P(z); W}: V\tens_{P(z)} W\rightarrow W$ and $r_{P(z); W}: W\tens_{P(z)} V\rightarrow W$ for $W$ in $\cC$ characterized by
\begin{equation*}
 \overline{l_{P(z); W}}(v\tens_{P(z)} w) = Y_W(v, z)w
\end{equation*}
and
\begin{equation*}
 \overline{r_{P(z); W}}(w\tens_{P(z)} v)=e^{z L(-1)} Y_W(v,-z)w
\end{equation*}
for $v\in V$, $w\in W$.

\item \textit{$P(z_1,z_2)$-associativity isomorphisms}: For $z_1, z_2\in\CC^\times$ such that $\vert z_1\vert>\vert z_2\vert>\vert z_1-z_2\vert>0$, and for $W_1$, $W_2$, $W_3$ in $\cC$, there is a natural isomorphism
\begin{equation*}
 \cA_{P(z_1,z_2); W_1,W_2,W_3}: W_1\tens_{P(z_1)}(W_2\tens_{P(z_2)} W_3)\rightarrow (W_1\tens_{P(z_1-z_2)} W_2)\tens_{P(z_2)} W_3
\end{equation*}
characterized by
\begin{equation*}
 \overline{\cA_{P(z_1,z_2); W_1, W_2, W_3}}(w_1\tens_{P(z_1)}(w_2\tens_{P(z_2)} w_3)=(w_1\tens_{P(z_1-z_2)} w_2)\tens_{P(z_2)} w_3
\end{equation*}
for $w_1\in W_1$, $w_2\in W_2$, $w_3\in W_3$. Here, the meaning of triple tensor products of elements of $V$-modules is as follows. We identify $w_1\tens_{P(z_1)}(w_2\tens_{P(z_2)} w_3)\in\overline{W_1\tens_{P(z_1)}(W_2\tens_{P(z_2)} W_3)}$ by its action as an element in the dual space of $(W_1\tens_{P(z_1)}(W_2\tens_{P(z_2)} W_3))'$:
\begin{equation*}
 \langle w', w_1\tens_{P(z_1)}(w_2\tens_{P(z_2)} w_3)\rangle =\sum_{h\in\CC} \langle w', w_1\tens_{P(z_1)}\pi_h(w_2\tens_{P(z_2)} w_3)\rangle
\end{equation*}
for any $w'\in (W_1\tens_{P(z_1)}(W_2\tens_{P(z_2)} W_3))'$, and $(w_1\tens_{P(z_1-z_2)} w_2)\tens_{P(z_2)} w_3$ is interpreted similarly. Clearly, these triple tensor products of elements only make sense if the corresponding sums converge absolutely; such analytic issues are the main reason why constructing associativity isomorphisms in a category of $V$-modules is generally difficult.

\item \textit{$P(z)$-braiding isomorphisms}: For any $z\in\CC^\times$ and $W_1$, $W_2$ objects of $\cC$, there is a natural isomorphism
\begin{equation*}
 \cR_{P(z); W_1,W_2}: W_1\tens_{P(z)} W_2\rightarrow W_2\tens_{P(-z)} W_1
\end{equation*}
characterized by
\begin{equation*}
 \overline{\cR_{P(z); W_1, W_2}}(w_1\tens_{P(z)} w_2)=e^{z L(-1)}\cY_{\boxtimes_{P(-z)}}(w_2, e^{\log z+\pi i})w_1
\end{equation*}
for $w_1\in W_1$, $w_2\in W_2$. Note here that $\log z+\pi i$ may not equal the branch $\log(-z)$. The inverse $P(z)$-braiding isomorphism is defined the same way, except one uses the branch $\log z-\pi i$ of logarithm for $-z$.
\end{itemize}
Additionally, these natural isomorphisms should satisfy analogues of the triangle, pentagon, and hexagon coherence conditions which may be found stated in the proof of \cite[Theorem 12.15]{HLZ8}.

In \cite{HLZ8}, it is shown how to obtain a braided tensor category from the vertex tensor category structure on $\cC$. One chooses the tensor product bifunctor $\tens$ to be $\tens_{P(1)}$ (actually, any $z\in\CC^\times$ would work). Then the unit, associativity, and braiding isomorphisms are as follows:
\begin{itemize}
 \item The unit object is $V$ and the unit isomorphisms $l$ and $r$ are $l_{P(1)}$ and $r_{P(1)}$, respectively.
 
 \item The natural associativity isomorphism $\cA$ is given by
 \begin{align*}
  \cA_{W_1,W_2,W_3} & =T_{r_2\to 1; W_1\tens W_2, W_3}\circ(T_{r_1-r_2\to 1; W_1,W_2}\tens_{P(r_2)} 1_{W_3})\circ\nonumber\\
  &\hspace{3em}\circ\cA_{P(r_1,r_2), W_1,W_2, W_3}\circ(1_{W_1}\tens_{P(r_1)} T_{1\to r_2; W_2,W_3})\circ T_{1\to r_1; W_1, W_2\tens W_3}
 \end{align*}
where $r_1, r_2$ are positive real numbers chosen so that $r_1>r_2>r_1-r_2>0$. The associativity isomorphisms do not depend on the choice of such $r_1$, $r_2$ (see for instance \cite[Proposition 3.32]{CKM}).

\item The natural braiding isomorphism is given by $\cR=T_{-1\to 1}\circ\cR_{P(1)}$.
\end{itemize}
The triangle, pentagon, and hexagon identities for this braided tensor category structure on $\cC$ follow from the corresponding coherence conditions on the vertex tensor category, as shown in \cite[Theorem 12.15]{HLZ8}.

In \cite[Section 3.6]{CKM}, the notion of \textit{vertex tensor functor} was introduced to describe functors which preserve vertex tensor category structure. In particular, let $\cC_1$ and $\cC_2$ be vertex tensor categories with $P(z)$-tensor product functors, parallel transport isomorphisms, $P(z)$-unit and braiding isomorphisms, and $P(z_1,z_2)$-associativity isomorphisms satisfying the appropriate coherence conditions; we do not necessarily assume that $\cC_1$ or $\cC_2$ is a module category for a vertex operator algebra. In this context, a vertex tensor functor is a triple $(\Phi, \lbrace J_{P(z)}\rbrace_{z\in\CC^\times}, \varphi)$ where $\Phi: \cC_1\rightarrow\cC_2$ is a functor, $J_{P(z)}:   \tens_{P(z)}\circ(\Phi\times\Phi)\rightarrow\Phi\circ\tens_{P(z)}$ is a natural isomorphism for all $z\in\CC^\times$, and $\varphi: V_2\rightarrow\Phi(V_1)$ is an isomorphism (here $V_1$ and $V_2$ are the unit objects of the vertex tensor categories $\cC_1$ and $\cC_2$ respectively). These isomorphisms satisfy the following compatibility conditions:
\begin{itemize}
 \item The natural isomorphisms $J_{P(z)}$ are compatible with parallel transport isomorphisms in the sense that the diagram
 \begin{equation}\label{partranscompat}
  \xymatrixcolsep{5pc}
  \xymatrix{
  \Phi(M_1)\tens_{P(z_1)} \Phi(M_2) \ar[d]^{J_{P(z_1); M_1,M_2}} \ar[r]^{T_{\gamma; \Phi(M_1),\Phi(M_2)}} & \Phi(M_1)\tens_{P(z_2)}\Phi(M_2) \ar[d]^{J_{P(z_2); M_1,M_2}} \\
  \Phi(M_1\tens_{P(z_1)} M_2) \ar[r]^{\Phi(T_{\gamma; M_1,M_2})} & \Phi(M_1\tens_{P(z_2)} M_2) \\
  }
 \end{equation}
 commutes for all objects $M_1$, $M_2$ in $\cC_1$ and all continuous paths $\gamma$ in $\CC^\times$ beginning at $z_1$ and ending at $z_2$.
 
 \item The isomorphisms $J_{P(z)}$ and $\varphi$ are compatible with unit isomorphisms in the sense that the diagrams
 \begin{equation}\label{unitcompat}
 \xymatrixcolsep{3.5pc}
 \xymatrix{
 V_2\boxtimes_{P(z)}\Phi(M) \ar[d]^{\varphi\boxtimes_{P(z)} 1_{\Phi(M)}} \ar[r]^(.6){l_{P(z); \Phi(M)}} & \Phi(M) \\
 \Phi(V_1)\boxtimes_{P(z)}\Phi(M) \ar[r]^(.55){J_{P(z); V_1, M}} & \Phi(V_1\tens_{P(z)} M) \ar[u]_{\Phi(l_{P(z);M})} \\
 }\hspace{2em}\vspace{6em}
 \xymatrix{
 \Phi(M)\boxtimes_{P(z)} V_2 \ar[d]^{1_{\Phi(M)}\boxtimes_{P(z)}\varphi} \ar[r]^(.6){r_{P(z);\Phi(M)}} & \Phi(M) \\
 \Phi(M)\boxtimes_{P(z)}\Phi(V_1) \ar[r]^(.55){J_{P(z); M,V_1}} & \Phi(M\tens_{P(z)} V_1) \ar[u]_{\Phi(r_{P(z); M})} \\
 }
\end{equation}
commute for any object $M$ in $\cC_1$ and $z\in\CC^\times$.

\item The natural isomorphisms $J_{P(z)}$ are compatible with associativity isomorphisms in the sense that the diagram
\begin{equation}\label{assoccompat}
  \xymatrixcolsep{8.5pc}
  \xymatrix{
  \Phi(M_1)\boxtimes_{P(z_1)}(\Phi(M_2)\boxtimes_{P(z_2)}\Phi(M_3)) \ar[d]^{1_{\Phi(M_1)}\boxtimes_{P(z_1)} J_{P(z_2); M_2,M_3}} \ar[r]^(.49){\cA_{P(z_1,z_2); \Phi(M_1),\Phi(M_2),\Phi(M_3)}} & (\Phi(M_1)\boxtimes_{P(z_1-z_2)}\Phi(M_2))\tens_{P(z_2)}\Phi(M_3) \ar[d]_{J_{P(z_1-z_2);M_1,M_2}\boxtimes_{P(z_2)} 1_{\Phi(M_3)}} \\
  \Phi(M_1)\boxtimes_{P(z_1)}\Phi(M_2\tens_{P(z_2)} M_3) \ar[d]^{J_{P(z_1);M_1,M_2\tens_{P(z_2)} M_3}} & \Phi(M_1\tens_{P(z_1-z_2)} M_2)\tens_{P(z_2)}\Phi(M_3) \ar[d]_{J_{P(z_2);M_1\tens_{P(z_1-z_2)} M_2, M_3}} \\
  \Phi(M_1\tens_{P(z_1)}(M_2\tens_{P(z_2)} M_3)) \ar[r]^{\Phi(\cA_{P(z_1,z_2); M_1,M_2,M_3})} & \Phi((M_1\tens_{P(z_1-z_2)} M_2)\tens_{P(z_2)} M_3) \\
  }
 \end{equation}
 commutes for any objects $M_1$, $M_2$, $M_3$ in $\cC_1$ and $z_1,z_2\in\CC^\times$ such that $\vert z_1\vert>\vert z_2\vert>\vert z_1-z_2\vert>0$.
 
 \item The natural isomorphisms $J_{P(z)}$ are compatible with braiding isomorphisms in the sense that the diagram
 \begin{equation}\label{braidcompat}
 \xymatrixcolsep{5.5pc}
 \xymatrix{
 \Phi(M_1)\tens_{P(z)}\Phi(M_2) \ar[d]^{J_{P(z); M_1,M_2}} \ar[r]^{\cR_{P(z); \Phi(M_1),\Phi(M_2)}} & \Phi(M_2)\tens_{P(-z)} \Phi(M_1) \ar[d]^{J_{P(-z); M_2,M_1}} \\
 \Phi(M_1\tens_{P(z)} M_2) \ar[r]^{\Phi(\cR_{P(z); M_1,M_2})} & \Phi(M_2\tens_{P(-z)} M_1) \\
 }
\end{equation}
commutes for any objects $M_1$, $M_2$ in $\cC_1$ and $z\in\CC^\times$.
\end{itemize}
If the $J_{P(z)}$ are natural transformations but not necessarily isomorphisms, we say that $(\Phi, \lbrace J_{P(z)}\rbrace_{z\in\CC^\times}, \varphi)$ is a \textit{lax} vertex tensor functor.

Given that vertex tensor category structure induces braided tensor category structure, we would like to know that vertex tensor functors induce braided tensor functors:
\begin{theo}\label{vrtxtobraidfunctor}
 If $(\Phi, \lbrace J_{P(z)}\rbrace_{z\in\CC^\times}, \varphi)$ is a (lax) vertex tensor functor between vertex tensor categories $\cC_1$ and $\cC_2$, then $(\Phi, J=J_{P(1)}, \varphi)$ is a (lax) braided tensor functor with respect to the induced braided tensor category structures on $\cC_1$ and $\cC_2$.
\end{theo}
\begin{proof}
 We just need to show that $J$ and $\varphi$ are compatible with the unit, associativity, and braiding isomorphisms in the braided tensor categories $\cC_1$ and $\cC_2$. Compatibility with the unit isomorphisms is precisely the commutativity of the diagrams in \eqref{unitcompat} for the case $z=1$, so it remains to consider the associativity and braiding isomorphisms.
 
 For the associativity isomorphisms, we have to verify that the diagram
 \begin{equation}\label{assoccompat1}
  \xymatrixcolsep{6pc}
  \xymatrix{
  \Phi(M_1)\boxtimes(\Phi(M_2)\boxtimes\Phi(M_3)) \ar[d]^{1_{\Phi(M_1)}\boxtimes J_{M_2,M_3}} \ar[r]^{\cA_{\Phi(M_1),\Phi(M_2),\Phi(M_3)}} & (\Phi(M_1)\boxtimes\Phi(M_2))\tens\Phi(M_3) \ar[d]^{J_{M_1,M_2}\boxtimes 1_{\Phi(M_3)}} \\
  \Phi(M_1)\boxtimes\Phi(M_2\tens M_3) \ar[d]^{J_{M_1,M_2\tens M_3}} & \Phi(M_1\tens M_2)\tens\Phi(M_3) \ar[d]^{J_{M_1\tens M_2, M_3}} \\
  \Phi(M_1\tens(M_2\tens M_3)) \ar[r]^{\Phi(\cA_{M_1,M_2,M_3})} & \Phi((M_1\tens M_2)\tens M_3) \\
  }
 \end{equation}
 commutes for any $M_1$, $M_2$, $M_3$ in $\cC_1$. For this we choose positive real numbers $r_1, r_2$ such that $r_1>r_2>r_1-r_2>0$. Now we use the following diagrams; here to save space we suppress object subscripts on morphisms and write $W_i=\Phi(M_i)$ for $i=1,2,3$:
 \begin{equation*}
  \xymatrixcolsep{4.65pc}
  \xymatrix{
  W_1\tens_{P(r_1)}(W_2\tens_{P(r_2)} W_3) \ar[d]^{\cA_{P(r_1,r_2)}} \ar[r]^{1\tens_{P(r_1)} T_{r_2\to 1}} & W_1\tens_{P(r_1)}(W_2\tens W_3) \ar[r]^{T_{r_1\to 1}} & W_1\tens(W_2\tens W_3) \ar[d]^{\cA} \\
  (W_1\tens_{P(r_1-r_2)} W_2)\tens_{P(r_2)} W_3 \ar[d]^{J_{P(r_1-r_2)}\tens_{P(r_2)} 1} \ar[r]^(.55){T_{r_1-r_2\to 1}\tens_{P(r_2)} 1} & (W_1\tens W_2)\tens_{P(r_2)} W_3 \ar[d]^{J\tens_{P(r_2)} 1} \ar[r]^{T_{r_2\to 1}} & (W_1\tens W_2)\tens W_3 \ar[d]^{J\tens 1} \\
 \Phi(M_1\tens_{P(r_1-r_2)} M_2)\tens_{P(r_2)} W_3 \ar[d]^{J_{P(r_2)}} \ar[r]^(.54){\Phi(T_{r_1-r_2\to 1})\tens_{P(r_2)} 1} & \Phi(M_1\tens M_2)\tens_{P(r_2)} W_3 \ar[d]^{J_{P(r_2)}} \ar[r]^{T_{r_2\to 1}} & \Phi(M_1\tens M_2)\tens W_3 \ar[d]^{J} \\
 \Phi((M_1\tens_{P(r_1-r_2)} M_2)\tens_{P(r_2)} M_3) \ar[r]^(.54){\Phi(T_{r_1-r_2\to 1}\tens_{P(r_2)} 1)} & \Phi((M_1\tens M_2)\tens_{P(r_2)} M_3) \ar[r]^{\Phi(T_{r_2\to 1})} & \Phi((M_1\tens M_2)\tens M_3) \\
  }
 \end{equation*}
and
\begin{equation*}
 \xymatrixcolsep{4.65pc}
 \xymatrix{
  W_1\tens_{P(r_1)}(W_2\tens_{P(r_2)} W_3) \ar[d]^{1\tens_{P(r_1)} J_{P(r_2)}} \ar[r]^{1\tens_{P(r_1)} T_{r_2\to 1}} & W_1\tens_{P(r_1)}(W_2\tens W_3) \ar[d]^{1\tens_{P(r_1)} J} \ar[r]^{T_{r_1\to 1}} & W_1\tens(W_2\tens W_3) \ar[d]^{1\tens J} \\
 W_1\tens_{P(r_1)} \Phi(M_2\tens_{P(r_2)} M_3) \ar[d]^{J_{P(r_1)}} \ar[r]^{1\tens_{P(r_1)}\Phi(T_{r_2\to 1})} & W_1\tens_{P(r_1)}\Phi(M_2\tens M_3) \ar[d]^{J_{P(r_1)}} \ar[r]^{T_{r_1\to 1}} & W_1\tens\Phi(M_2\tens M_3) \ar[d]^{J}\\
 \Phi(M_1\tens_{P(r_1)}(M_2\tens_{P(r_2)} M_3)) \ar[d]^{\Phi(\cA_{P(r_1,r_2)})} \ar[r]^(.53){\Phi(1\tens_{P(r_1)} T_{r_2\to 1})} & \Phi(M_1\tens_{P(r_1)}(M_2\tens M_3)) \ar[r]^{\Phi(T_{r_1\to 1})} & \Phi(M_1\tens(M_2\tens M_3)) \ar[d]^{\Phi(\cA)}\\
 \Phi((M_1\tens_{P(r_1-r_2)} M_2)\tens_{P(r_2)} M_3) \ar[r]^(.54){\Phi(T_{r_1-r_2\to 1}\tens_{P(r_2)} 1)} & \Phi((M_1\tens M_2)\tens_{P(r_2)} M_3) \ar[r]^{\Phi(T_{r_2\to 1})} & \Phi((M_1\tens M_2)\tens M_3) \\
 }
\end{equation*}
The diagrams commute by definition of the associativity isomorphisms in the braided tensor categories $\cC_1$ and $\cC_2$, the naturality of parallel transport isomorphisms and the $J_{P(z)}$, and the compatibility of the $J_{P(z)}$ with parallel transport isomorphisms. Now the commutativity of \eqref{assoccompat1} follows from the compatibility of the $J_{P(z)}$ with the $P(r_1,r_2)$-associativity isomorphisms in $\cC_1$ and $\cC_2$.
 
 For the braiding isomorphisms, we have to verify that
 \begin{equation*}
  J_{M_2,M_1}\circ\cR_{\Phi(M_1),\Phi(M_2)} =\Phi(\cR_{M_1,M_2})\circ J_{M_1,M_2}
 \end{equation*}
for any $M_1$, $M_2$ in $\cC_1$. This follows from  the diagram
\begin{equation*}
 \xymatrixcolsep{5.2pc}
 \xymatrix{
 \Phi(M_1)\tens\Phi(M_2) \ar[d]^{J_{M_1,M_2}} \ar[r]^(.47){\cR_{P(1);\Phi(M_1),\Phi(M_2)}} & \Phi(M_2)\tens_{P(-1)}\Phi(M_2) \ar[d]^{J_{P(-1); M_2,M_1}} \ar[r]^(.54){T_{-1\to 1;\Phi(M_2),\Phi(M_1)}} & \Phi(M_2)\tens\Phi(M_1) \ar[d]^{J_{M_2,M_1}} \\
 \Phi(M_1\tens M_2) \ar[r]^{\Phi(\cR_{P(1); M_1,M_2})} & \Phi(M_2\tens_{P(-1)} M_1) \ar[r]^{\Phi(T_{-1\to 1; M_1,M_2})} & \Phi(M_2\tens M_1) \\
 }
\end{equation*}
which commutes by compatibility of $J$ with the $P(1)$-braiding and parallel transport isomorphisms, and the definition $\cR=T_{-1\to 1}\circ\cR_{P(1)}$ in $\cC_1$ and $\cC_2$.
\end{proof}

\begin{rema}
 Conversely, given a (lax) braided tensor functor $(\Phi, J, \varphi)$ between vertex tensor categories $\cC_1$ and $\cC_2$, one could attempt to extend to a (lax) vertex tensor functor by defining
 \begin{equation*}
  J_{P(z)}=\Phi(T_{1\to z})\circ J\circ T_{z\to 1}
 \end{equation*}
for $z\in\CC^\times$. In \cite[Theorem 3.68]{CKM}, it was shown that this construction indeed yields a vertex tensor functor in the special case that $\Phi$ is the induction tensor functor from a subcategory of $V$-modules to $A$-modules, where $V$ is a suitable vertex operator subalgebra of $A$. Actually, the proof in \cite{CKM} works generally, except that proving compatibility of $J_{P(z)}$ with the parallel transport isomorphisms needs the identity
\begin{equation*}
 T_{\gamma; W_1, W_2}  = \cR_{P(-1); W_2, W_1}\circ\cR_{P(1); W_1, W_2}
\end{equation*}
where $\gamma: [0,1]\rightarrow\CC^\times$ is the continuous path given by $\gamma(t)=e^{-2\pi i t}$. This identity always holds in vertex tensor categories constructed from vertex operator algebra modules, so it would make sense to include it as a coherence condition in the definition of vertex tensor category.
\end{rema}

\subsection{Group-module categories modified by abelian \texorpdfstring{$3$}{3}-cocycles}\label{sec:gpmodcats}

The notion of abelian intertwining algebra in \cite{DL} is a generalization of the notion of vertex operator algebra using the cohomology of abelian groups introduced by Eilenberg and MacLane (\cite{Ei}, \cite{MacL}, \cite{EM1}, \cite{EM2}), specifically the notion of normalized abelian $3$-cocycle. Let $A$ be an abelian group.
\begin{defi}
 An \textit{abelian $3$-cocycle} on $A$ with values in $\CC^\times$ is a pair of functions
 \begin{equation*}
  F: A\times A\times A\rightarrow\CC^\times
 \end{equation*}
and
\begin{equation*}
 \Omega: A\times A\rightarrow\CC^\times
\end{equation*}
which satisfies the following conditions:
\begin{enumerate}
 \item The \textit{pentagon identity}:
 \begin{equation*}
  F(\alpha_1,\alpha_2,\alpha_3)F(\alpha_1, \alpha_2+\alpha_3,\alpha_4)F(\alpha_2,\alpha_3,\alpha_4)=F(\alpha_1,\alpha_2,\alpha_3+\alpha_4)F(\alpha_1+\alpha_2,\alpha_3,\alpha_4)
 \end{equation*}
for all $\alpha_1,\alpha_2,\alpha_3,\alpha_4\in A$.

\item The \textit{hexagon identities}:
\begin{equation*}
 F(\alpha_1,\alpha_2,\alpha_3)\Omega(\alpha_1+\alpha_2,\alpha_3)F(\alpha_3,\alpha_1,\alpha_2) =\Omega(\alpha_2,\alpha_3)F(\alpha_1,\alpha_3,\alpha_2)\Omega(\alpha_1,\alpha_3)
\end{equation*}
and
\begin{equation*}
 F(\alpha_1,\alpha_2,\alpha_3)^{-1}\Omega(\alpha_1,\alpha_2+\alpha_3)F(\alpha_2,\alpha_3,\alpha_1)^{-1} = \Omega(\alpha_1,\alpha_2)F(\alpha_2,\alpha_1,\alpha_3)^{-1}\Omega(\alpha_1,\alpha_3)
\end{equation*}
for all $\alpha_1, \alpha_2,\alpha_3\in A$.
\end{enumerate}
The $3$-cocycle $(F,\Omega)$ is \textit{normalized} if
\begin{equation*}
 F(\alpha_1,\alpha_2, 0)=F(\alpha_1,0,\alpha_3)=F(0,\alpha_2,\alpha_3)=1
\end{equation*}
for $\alpha_1,\alpha_2,\alpha_3\in A$.
\end{defi}

\begin{rema}\label{OmegaNormal}
 It is clear from the hexagon identities that a normalized abelian $3$-cocycle also satisfies
\begin{equation*}
 \Omega(\alpha_1,0)=\Omega(0,\alpha_2)=1
\end{equation*}
for $\alpha_1,\alpha_2\in A$.
\end{rema}

\begin{rema}\label{quadform}
 The pentagon and hexagon identities for an abelian $3$-cocycle imply that the function $\alpha\mapsto\Omega(\alpha,\alpha)$ is a quadratic form with associated symmetric bimultiplicative form $(\alpha_1,\alpha_2)\mapsto\Omega(\alpha_1,\alpha_2)\Omega(\alpha_2,\alpha_1)$. This quadratic form is called the trace, and remarkably, the cohomology class of the abelian $3$-cocycle $(F,\Omega)$ is completely determined by the trace \cite{MacL}.
\end{rema}

We may view $A$ as a discrete topological group, so that its dual group $\widehat{A}$ of homomorphisms $\chi: A\rightarrow U(1)$ is a compact topological group. We shall be interested in compact groups $G$ which contain $\widehat{A}$ as a central subgroup. If $M$ is any finite-dimensional irreducible unitary representation of such a group $G$, then the action of $\widehat{A}$ on $M$ is given by a character of $\widehat{A}$, which is to say, an element of $A$. Thus the category $\rep\,G$ of finite-dimensional continuous representations of $G$ has an $A$-grading
\begin{equation*}
 \rep\,G=\bigoplus_{\alpha\in A} \rep^\alpha\,G
\end{equation*}
where for any representation $M$ in $\rep^\alpha\,G$, a character $\chi\in\widehat{A}$ acts on $M$ by the scalar $\chi(\alpha)$.

Recall that $\rep\,G$ is a rigid symmetric tensor category with unit object and duals, as well as unit, associativity, and braiding isomorphisms, inherited from the rigid symmetric tensor category of finite-dimensional vector spaces over $\mathbb{C}$. Given a normalized abelian $3$-cocycle $(F,\Omega)$ on $A$, we will use $\rep_{A,F,\Omega}\,G$ to denote the following modification of the tensor category structure on $\rep\,G$:
\begin{defi}
 $\rep_{A,F,\Omega}\,G$ is the braided tensor category whose objects, morphisms, unit object, and unit isomorphisms are the same as in $\rep\,G$. For objects $M_1$ in $\rep^{\alpha_1}\,G$, $M_2$ in $\rep^{\alpha_2}\,G$ and $M_3$ in $\rep^{\alpha_3}\,G$, the associativity isomorphism in $\rep_{A,F,\Omega}\,G$ is given by
 \begin{equation*}
  \cA_{M_1,M_2,M_3}(m_1\otimes(m_2\otimes m_3))=F(\alpha_1,\alpha_2,\alpha_3)^{-1}\left((m_1\otimes m_2)\otimes m_3\right)
 \end{equation*}
for $m_i\in M_i$. For objects $M_1$ in $\rep^{\alpha_1}\,G$ and $M_2\in\rep^{\alpha_2}\,G$, the braiding isomorphism in $\rep_{A,F,\Omega}\,G$ is given by
\begin{equation*}
 \cR_{M_1,M_2}(m_1\otimes m_2)=\Omega(\alpha_1,\alpha_2)^{-1}(m_2\otimes m_1)
\end{equation*}
for $m_i\in M_i$.
\end{defi}
\begin{rema}
 The pentagon identity for $F$ guarantees that the associativity isomorphisms in $\rep_{A,F,\Omega}\,G$ satisfy the pentagon axiom for a tensor category, and the hexagon identities for $F$ and $\Omega$ guarantee that the braiding isomorphisms in $\rep_{A,F,\Omega}\,G$ satisfy the hexagon axioms. The normalization of $F$ guarantees that the triangle axiom is satisfied in $\rep_{A,F,\Omega}\,G$. 
\end{rema}
\begin{rema}
The braided tensor category structure on $\mathrm{Rep}_{A,F,\Omega}\,G$ is independent of the cohomology class of the normalized abelian $3$-cocycle $(F,\Omega)$. In particular, $\mathrm{Rep}_{A,F,\Omega}\,G$ is isomorphic to $\rep\,G$ as braided tensor category if $(F,\Omega)$ is a normalized abelian $3$-coboundary, that is,
 \begin{equation*}
  F(\alpha_1,\alpha_2,\alpha_3) = \eta(\alpha_2,\alpha_3)\eta(\alpha_1+\alpha_2,\alpha_3)^{-1}\eta(\alpha_1,\alpha_2+\alpha_3)\eta(\alpha_1,\alpha_2)^{-1}
 \end{equation*}
and 
\begin{equation*}
 \Omega(\alpha_1,\alpha_2)=\eta(\alpha_1,\alpha_2)^{-1}\eta(\alpha_2,\alpha_1)
\end{equation*}
for some function $\eta: A\times A\rightarrow\CC^\times$ which satisfies $\eta(0,\alpha)=\eta(\alpha,0)=1$. In this case, the identity functor gives an isomorphism of categories between $\rep\,G$ and $\rep_{A,F,\Omega}\,G$, and the coboundary conditions guarantee that $\eta$ defines functorial isomorphisms compatible with the associativity and braiding isomorphisms of $\rep\,G$ and $\rep_{A,F,\Omega}\,G$.
\end{rema}

The braided tensor category $\rep_{A,F,\Omega}\,G$ is rigid, with the dual of a $G$-module $M$ given by $M^*$. However, it is necessary to modify either the evaluation or coevaluation morphism of $M$ by $F$. Specifically, we will take the coevaluation morphism for $M$ in $\rep^{\alpha}\,G$ to be
\begin{align*}
 i_M: \CC & \rightarrow M\otimes M^*\nonumber\\
  1 & \mapsto \sum_{i=1}^{\mathrm{dim}\,M} m_i\otimes m_i'
\end{align*}
where $\lbrace m_i\rbrace $ is a basis for $M$ and $\lbrace m_i'\rbrace$ is the corresponding dual basis for $M^*$; and we will take the evaluation morphism for $M$ to be
\begin{align*}
 e_M: M^*\otimes M &\rightarrow \mathbb{C}\nonumber\\
 m'\otimes m & \mapsto F(\alpha,-\alpha,\alpha)^{-1}\langle m',m\rangle.
\end{align*}
\begin{rema}\label{Frigidity}
 Verifying that the rigidity axioms for a tensor category are satisfied with this choice of evaluation and coevaluation amounts to the observation that $M^*$ is an object of $\rep^{-\alpha}\,G$, as well as the identity
 \begin{equation*}
  F(-\alpha,\alpha,-\alpha)=F(\alpha,-\alpha,\alpha)^{-1},
 \end{equation*}
for $\alpha\in A$, which follows from the pentagon identity and normalization of $F$.
\end{rema}

The rigid tensor category $\rep_{A,F,\Omega}\,G$ is also a ribbon category, with twist $\theta_M$ for $M$ an object of $\rep^\alpha\,G$ given by the scalar $\Omega(\alpha,\alpha)^{-1}$. The balancing equation
\begin{equation*}
 \theta_{M_1\otimes M_2}=\cR_{M_2,M_1}\circ\cR_{M_1,M_2}\circ(\theta_{M_1}\otimes\theta_{M_2})
\end{equation*}
follows from Remark \ref{quadform}. The identity $\theta_\CC=1_\CC$ is immediate from Remark \ref{OmegaNormal}, and the relation $\theta_{M}^*=\theta_{M^*}$ amounts to the identity
\begin{equation*}
 \Omega(\alpha,\alpha)=\Omega(-\alpha,-\alpha)
\end{equation*}
for $\alpha\in A$, which is immediate because $\alpha\mapsto\Omega(\alpha,\alpha)$ is a quadratic form (Remark \ref{quadform} again).

Recall that in a braided ribbon category, the trace $\mathrm{Tr}\,f$ of an endomorphism $f: M\rightarrow M$ is the endomorphism of the unit object $\vac$ given by the composition
\begin{equation*}
 \vac\xrightarrow{i_{M}} M\otimes M^*\xrightarrow{(\theta_M\circ f)\otimes 1_{M^*}} M\otimes M^*\xrightarrow{\cR_{M,M^*}} M^*\otimes M\xrightarrow{e_M} \vac.
\end{equation*}
The categorical dimension of $M$ is then $\mathrm{Tr}\,1_M$. For a $G$-module $M$ in $\rep^\alpha\,G$, the categorical dimension in $\rep_{A,F,\Omega}\,G$ is thus
\begin{align*}
 &\left(\Omega(\alpha,\alpha)  \Omega(\alpha,-\alpha)F(\alpha,-\alpha,\alpha)\right)^{-1}\dim_\CC\,M\nonumber\\
 &\hspace{4em}= \left(F(-\alpha,\alpha,\alpha)F(\alpha,-\alpha,\alpha)^{-1}\Omega(\alpha,0)F(-\alpha,\alpha,\alpha)^{-1} F(\alpha,-\alpha,\alpha)\right)^{-1}\dim_\CC\,M\nonumber\\
 & \hspace{4em}=\dim_\CC\,M,
\end{align*}
using the second hexagon identity in the case $\alpha_1=\alpha_3=\alpha$, $\alpha_2=-\alpha$. Thus categorical dimensions in $\rep_{A,F,\Omega}\,G$ agree with the ordinary dimensions of $G$-modules in $\rep\,G$.

\subsection{Abelian intertwining algebras and automorphisms}\label{sec:Abintwalg}

Here we recall the notion of abelian intertwining algebra from \cite{DL} and discuss automorphism groups of abelian intertwining algebras. Let $A$ be an abelian group and $(F,\Omega)$ a normalized abelian $3$-cocycle on $A$ with values in $\CC^\times$. Define the functions
\begin{equation*}
 q: A\rightarrow\CC/\ZZ,\hspace{3em}b: A\times A\rightarrow\CC/\ZZ
\end{equation*}
by
\begin{equation*}
 \Omega(\alpha,\alpha)=e^{2\pi i q(\alpha)},\hspace{3em}\Omega(\alpha_1,\alpha_2)\Omega(\alpha_2,\alpha_1)=e^{2\pi i b(\alpha_1,\alpha_2)},
\end{equation*}
so that $b$ is symmetric and bilinear. For convenience, we also (non-uniquely) fix a lift
\begin{equation*}
 \hat{b}: A\times A\rightarrow \CC
\end{equation*}
such that
\begin{equation*}
 b(\alpha_1,\alpha_2)=\hat{b}(\alpha_1,\alpha_2)+\ZZ
\end{equation*}
for any $\alpha_1,\alpha_2\in A$. We will also use the function
\begin{equation*}
 B: A\times A\times A\rightarrow\CC^\times
\end{equation*}
given by
\begin{equation*}
 B(\alpha_1,\alpha_2,\alpha_3)=F(\alpha_1,\alpha_2,\alpha_3)\Omega(\alpha_1,\alpha_2)F(\alpha_2,\alpha_1,\alpha_3)^{-1}.
\end{equation*}
for $\alpha_1,\alpha_2,\alpha_3\in A$.

\begin{defi}
 An \textit{abelian intertwining algebra} associated to an abelian group $A$ with normalized abelian $3$-cocycle $(F,\Omega)$ is an $A\times\CC$-graded vector space 
 \begin{equation*}
  V=\bigoplus_{\alpha\in A, n\in\CC} V^\alpha_{(n)}
 \end{equation*}
equipped with a \textit{vertex operator}
\begin{align*}
 Y: V & \rightarrow (\Endo V)\lbrace x\rbrace\\
 v & \mapsto Y(v,x)=\sum_{n\in\CC} v_n\,x^{-n-1}
\end{align*}
and two distinguished vectors: $\vac\in V^0_{(0)}$ called the \textit{vacuum} and $\omega\in V^0_{(2)}$ called the \textit{conformal vector}. For $\alpha\in A$ we use the notation $V^\alpha=\bigoplus_{n\in\CC} V^\alpha_{(n)}$, and for $n\in\CC$ we use the notation $V_{(n)}=\bigoplus_{\alpha\in A} V^\alpha_{(n)}$. The data satisfy the following axioms:
\begin{enumerate}
 \item The \textit{grading-restriction conditions}: For any $\alpha\in A$ and $R\in\RR$, $V^\alpha_{(n)}=0$ for all but finitely many $n$ with $\mathrm{Re}\,n\leq R$. (Optional additional grading-restriction condition: 
 \begin{equation*}
  V^\alpha=\bigoplus_{n\in-q(\alpha)} V^\alpha_{(n)}
 \end{equation*}
for any $\alpha\in A$.)

\item Lower truncation and compatibility of $Y$ with the $A$-grading: For any $\alpha_1,\alpha_2\in A$ and $v_1\in V^{\alpha_1}, v_2\in V^{\alpha_2}$,
\begin{equation*}
 x^{\hat{b}(\alpha_1,\alpha_2)} Y(v_1,x)v_2\in V^{\alpha_1+\alpha_2}((x)).
\end{equation*}

\item The \textit{vacuum} and \textit{creation properties}: For any $v\in V$, $Y(\vac, x)v=v$ and $Y(v,x)\vac\in V[[x]]$ with constant term $v$.

\item The \textit{Jacobi identity}: For any $\alpha_1,\alpha_2,\alpha_3\in A$ and $v_1\in V^{\alpha_1}$, $v_2\in V^{\alpha_2}$, $v_3\in V^{\alpha_3}$,
\begin{align*}
 x_0^{-1} & \left(\frac{x_1-x_2}{x_0}\right)^{\hat{b}(\alpha_1,\alpha_2)}\delta\left(\frac{x_1-x_2}{x_0}\right) Y(v_1,x_1)Y(v_2,x_2)v_3 \nonumber\\
 &\hspace{3em} -B(\alpha_1,\alpha_2,\alpha_3)\, x_0^{-1}\left(\dfrac{x_2-x_1}{e^{\pi i} x_0}\right)^{\hat{b}(\alpha_1,\alpha_2)} \delta\left(\frac{x_2-x_1}{-x_0}\right) Y(v_2,x_2)Y(v_1,x_1)v_3\nonumber\\
 & = F(\alpha_1,\alpha_2,\alpha_3)\, x_1^{-1}\left(\frac{x_2+x_0}{x_1}\right)^{\hat{b}(\alpha_1,\alpha_3)}\delta\left(\frac{x_2+x_0}{x_1}\right) Y(Y(v_1,x_0)v_2,x_2)v_3.
\end{align*}

\item The Virasoro relations: If we set $Y(\omega, x)=\sum_{n\in\ZZ} L(n) x^{-n-2}$ (note that this is indeed the form of $Y(\omega, x)$ by the normalization of $\Omega$ and compatibility of $Y$ with the $A$-grading), then
\begin{equation*}
 [L(m),L(n)]=(m-n)L(m+n)+ c \,\dfrac{m^3-m}{12}\delta_{m+n,0}  1_V,
\end{equation*}
for any $m,n\in\ZZ$, where $c\in\CC$ is called the \textit{central charge} of $V$. Moreover, $V_{(n)}$ is the $L(0)$-eigenspace of $V$ with eigenvalue $n$; for $v\in V_{(n)}$, $n$ is denoted $\mathrm{wt}\,v$ and called the \textit{conformal weight} of $v$.

\item The \textit{$L(-1)$-derivative property}:
\begin{equation*}
 Y(L(-1)v, x)=\dfrac{d}{dx}Y(v,x)
\end{equation*}
for any $v\in V$.
\end{enumerate}
\end{defi}

\begin{rema}
 The definition of abelian intertwining algebra associated to $(F,\Omega)$ does not depend on the choice of lift $\hat{b}$, as replacing $\hat{b}(\alpha_1,\alpha_2)$ with the same plus an integer does not change any of the formulas in compatibility of $Y$ with the $A$-grading or the Jacobi identity.
\end{rema}

\begin{rema}
 Because the cocycle in the definition of abelian intertwining algebra is normalized, the subspace $V^0$ is a vertex operator algebra (although possibly not $\ZZ$-graded if the optional grading restriction condition in the definition fails to hold), and each $V^\alpha$ for $\alpha\in A$ is a $V^0$-module. Even more, $Y\vert_{V^{\alpha_1}\otimes V^{\alpha_2}}$ is a $V^0$-intertwining operator of type $\binom{V^{\alpha_1+\alpha_2}}{V^{\alpha_1}\,V^{\alpha_2}}$ for all $\alpha_1,\alpha_2\in A$.
\end{rema}

\begin{exam}
 If we take an abelian group $A$ with $F=1$ and $\Omega=1$, then an abelian intertwining algebra associated to $(F,\Omega)$, together with the optional additional grading restriction condition, is a strongly $A$-graded conformal vertex algebra in the sense of \cite{HLZ1}.
\end{exam}

\begin{exam}\label{exam:VOSA}
 Suppose we take $A=\ZZ/2\ZZ$, $F=1$, and
 \begin{equation*}
  \Omega(i_1+2\ZZ, i_2+2\ZZ) =(-1)^{i_1 i_2}
 \end{equation*}
 for $i_1,i_2\in\ZZ$. Then $(F,\Omega)$ is a normalized abelian $3$-cocycle, and an abelian intertwining algebra associated to $(F,\Omega)$ is a vertex operator superalgebra. If we include the optional additional grading restriction condition, we get a vertex operator superalgebra of ``correct statistics'' in the terminology of \cite{CKL}, that is, $V^\even$ is $\ZZ$-graded and $V^\odd$ is $(\frac{1}{2}+\ZZ)$-graded.
\end{exam}

We will need to use the following associativity property of the vertex operator for an abelian intertwining algebra, which is a consequence of the Jacobi identity (see \cite[Remark 12.31]{DL}): for $v_1\in V^{\alpha_1}$, $v_2\in V^{\alpha_2}$, $v_3\in V^{\alpha_3}$, and $v'\in V'$ (where $V'=\bigoplus_{\alpha\in A, n\in\CC} (V^{\alpha}_{(n)})^*$ is the graded dual of $V$),
\begin{equation*}
 \langle v', Y(v_1, z_1)Y(v_2, z_2)v_3\rangle (z_1-z_2)^{\hat{b}(\alpha_1,\alpha_2)} z_1^{\hat{b}(\alpha_1,\alpha_3)} z_2^{\hat{b}(\alpha_2,\alpha_3)}
\end{equation*}
and
\begin{equation*}
 F(\alpha_1,\alpha_2,\alpha_3)\langle v', Y(Y(v_1,z_1-z_2)v_2,z_2)v_3\rangle (z_1-z_2)^{\hat{b}(\alpha_1,\alpha_2)} z_1^{\hat{b}(\alpha_1,\alpha_3)} z_2^{\hat{b}(\alpha_2,\alpha_3)}
\end{equation*}
are the expansions of a common rational function, with poles possible only at $z_1=0$, $z_2 = 0$, and $z_1=z_2$, in the regions $\vert z_1\vert>\vert z_2\vert>0$ and $\vert z_2\vert>\vert z_1-z_2\vert>0$, respectively. If we use $f(z_1,z_2)$ to represent this common rational function (which of course also depends on $v_1$, $v_2$, $v_3$, and $v'$), then we have the following equalities of multivalued analytic functions:
\begin{equation*}
 \langle v', Y(v_1,z_1)Y(v_2, z_2)v_3\rangle = z_1^{-\hat{b}(\alpha_1,\alpha_2)-\hat{b}(\alpha_1,\alpha_3)} z_2^{-\hat{b}(\alpha_2,\alpha_3)}\left(1-\frac{z_2}{z_1}\right)^{-\hat{b}(\alpha_1,\alpha_2)} f(z_1,z_2)
\end{equation*}
on the region $\vert z_1\vert>\vert z_2\vert>0$, and
\begin{align*}
 \langle v', Y(Y(v_1, & z_1-z_2)v_2,z_2)v_3\rangle\nonumber\\
 &= F(\alpha_1,\alpha_2,\alpha_3)^{-1} (z_1-z_2)^{-\hat{b}(\alpha_1,\alpha_2)} z_2^{-\hat{b}(\alpha_1,\alpha_3)-\hat{b}(\alpha_2,\alpha_3)} \left(1+\frac{z_1-z_2}{z_2}\right)^{-\hat{b}(\alpha_1,\alpha_3)} f(z_1,z_2)
\end{align*}
on the region $\vert z_2\vert>\vert z_1-z_2\vert>0$. 

We would like to compare the branches of these multivalued functions obtained by specializing complex powers of $z_1$, $z_2$, and $z_1-z_2$ using the branch of logarithm $\log z = \log\vert z\vert +i\arg{z}$ for which $0\leq\arg{z}<2\pi$. To this end, for $z_1,z_2\in\CC^\times$ for which $\vert z_1\vert>\vert z_2\vert$, we write
\begin{equation*}
 \log(z_1-z_2) = \log z_1 +\log\left(1-\frac{z_2}{z_1}\right)+2\pi i p_{z_1,z_2},
\end{equation*}
where $\log\left(1-\frac{z_2}{z_1}\right)$ represents the standard series of $\log(1-x)$ centered at $x=0$ specialized to $x=z_2/z_1$, and   $p_{z_1,z_2}$ is some integer. Note that if also $\vert z_2\vert>\vert z_1-z_2\vert>0$ we then have
\begin{equation*}
 \log z_1=\log(z_2-(z_2-z_1))=\log z_2 +\log\left(1+\frac{z_1-z_2}{z_2}\right)+2\pi i p_{z_2,z_2-z_1}.
\end{equation*}
With this notation, if $\vert z_1\vert>\vert z_2\vert>\vert z_1-z_2\vert>0$, we can calculate:
\begin{align*}
 \langle v', Y(v_1, &  e^{\log z_1})Y(v_2, e^{\log z_2})v_3\rangle = (e^{\log z_1})^{-\hat{b}(\alpha_1,\alpha_2)-\hat{b}(\alpha_1,\alpha_3)} (e^{\log z_2})^{-\hat{b}(\alpha_2,\alpha_3)}\left(1-\frac{z_2}{z_1}\right)^{-\hat{b}(\alpha_1,\alpha_2)} f(z_1,z_2)\nonumber\\
 & = (e^{\log(z_1-z_2)})^{-\hat{b}(\alpha_1,\alpha_2)} e^{2\pi i p_{z_1,z_2} \hat{b}(\alpha_1,\alpha_2)} (e^{\log z_2})^{-\hat{b}(\alpha_1,\alpha_3)}\left(1+\frac{z_1-z_2}{z_2}\right)^{-\hat{b}(\alpha_1,\alpha_3)}\cdot\nonumber\\
 &\hspace{5em} \cdot e^{-2\pi i p_{z_2,z_2-z_1}\hat{b}(\alpha_1,\alpha_3)} (e^{\log z_2})^{-\hat{b}(\alpha_2,\alpha_3)} f(z_1,z_2)\nonumber\\
 & =\left(\Omega(\alpha_1,\alpha_2)\Omega(\alpha_2,\alpha_1)\right)^{p_{z_1,z_2}}\left(\Omega(\alpha_1,\alpha_3)\Omega(\alpha_3,\alpha_1)\right)^{-p_{z_2,z_2-z_1}}\cdot\nonumber\\
 &\hspace{5em} \cdot F(\alpha_1,\alpha_2,\alpha_3)\langle v', Y(Y(v_1, e^{\log(z_1-z_2)})v_2, e^{\log z_2})v_3\rangle .
\end{align*}
We state the result of this calculation as a proposition:
\begin{propo}\label{abintwalgassoc}
 If $z_1,z_2\in\CC^\times$ satisfy $\vert z_1\vert>\vert z_2\vert>\vert z_1-z_2\vert>0$, then
 \begin{align*}
   \langle v', Y(Y(v_1, & e^{\log(z_1-z_2)})v_2, e^{\log z_2})v_3 =\cA_{z_1,z_2}(\alpha_1,\alpha_2,\alpha_3) \langle v', Y(v_1, e^{\log z_1})Y(v_2,e^{\log z_2})v_3\rangle,  
 \end{align*}
 where
 \begin{equation*}
  \cA_{z_1,z_2}(\alpha_1,\alpha_2,\alpha_3) = \left(\Omega(\alpha_1,\alpha_2)\Omega(\alpha_2,\alpha_1)\right)^{-p_{z_1,z_2}}\left(\Omega(\alpha_1,\alpha_3)\Omega(\alpha_3,\alpha_1)\right)^{p_{z_2,z_2-z_1}} F(\alpha_1,\alpha_2,\alpha_3)^{-1},
 \end{equation*}
for any $v_1\in V^{\alpha_1}$, $v_2\in V^{\alpha_2}$, $v_3\in V^{\alpha_3}$, and $v'\in V'$.
\end{propo}

\begin{rema}\label{abintwalgassocreal}
 If $r_1$ and $r_2$ are positive real numbers satisfying $r_1>r_2>r_1-r_2>0$, then it is clear that $p_{r_1,r_2}=p_{r_2,r_2-r_1} = 0$. Thus we have $\cA_{r_1,r_2}(\alpha_1,\alpha_2,\alpha_3)=F(\alpha_1,\alpha_2,\alpha_3)^{-1}$ in this case.
\end{rema}

In addition to the above associativity result for $Y$, we will need the following skew-symmetry result, which can be found in \cite[Lemma 2.35]{Ca} (see also the proof of \cite[Proposition 2.6]{GL}):
\begin{propo}\label{abintalgskewsym}
 For $v_1\in V^{\alpha_1}$ and $v_2\in V^{\alpha_2}$,
 \begin{equation*}
  Y(v_1,x)v_2=\Omega(\alpha_1,\alpha_2) e^{x L(-1)}Y(v_2, e^{\pi i} x) v_1.
 \end{equation*}
\end{propo}
Iterating skew-symmetry $2p$ times for $p\in\ZZ$ yields the following corollary:
\begin{corol}\label{abintalgmono}
 For $v_1\in V^{\alpha_1}$ and $v_2\in V^{\alpha_2}$,
 \begin{equation*}
  Y(v_1,x)v_2=\left(\Omega(\alpha_1,\alpha_2)\Omega(\alpha_2,\alpha_1)\right)^{p} Y(v_1, e^{2\pi i p} x)v_2.
 \end{equation*}
\end{corol}

We will also need to consider products of three vertex operators. The following proposition is a special case of \cite[Theorem 12.33]{DL}, which generalizes \cite[Proposition 3.5.1]{FHL} to abelian intertwining algebras:
\begin{propo}\label{threevrtxops}
 For any $v_1\in V^{\alpha_1}$, $v_2\in V^{\alpha_2}$, $v_3\in V^{\alpha_3}$, $v_4\in V^{\alpha_4}$, and $v'\in V'$, the series
 \begin{equation*}
  \langle v', Y(v_1, z_1)Y(v_2,z_2)Y(v_3,z_3)v_4\rangle 
 \end{equation*}
converges absolutely in the region $\vert z_1\vert>\vert z_2\vert>\vert z_3\vert>0$ and can be extended to a multivalued analytic function of the form
 \begin{equation*}
  \frac{f(z_1,z_2,z_3)}{z_1^{\hat{b}(\alpha_1,\alpha_4)}z_2^{\hat{b}(\alpha_2,\alpha_4)}z_3^{\hat{b}(\alpha_3,\alpha_4)} (z_1-z_2)^{\hat{b}(\alpha_1,\alpha_2)} (z_1-z_3)^{\hat{b}(\alpha_1,\alpha_3)} (z_2-z_3)^{\hat{b}(\alpha_2,\alpha_3)}}
 \end{equation*}
where $f(z_1,z_2,z_3)$ is a rational function with poles possible only at $z_i=0$ and $z_i=z_j$ for $1\leq i,j\leq 3$ and $i\neq j$.
\end{propo}

The results of this paper will apply to abelian intertwining algebras that are simple in the following sense:
\begin{defi}
 We say that an abelian intertwining algebra $V$ is \textit{simple} if the only $A$-graded subspaces closed under the actions of $v_n$ for all $v\in V$, $n\in\CC$ are $0$ and $V$.
\end{defi}

\begin{propo}\label{Valphasimple}
 If $V$ is a simple abelian intertwining algebra, then the $V^0$-modules $V^\alpha$ are simple for all $\alpha\in A$. In particular, $V^0$ is a simple vertex operator algebra.
\end{propo}
\begin{proof}
 Fix any non-zero $v\in V^\alpha$. It is sufficient to show that $V^\alpha=\mathrm{span}\lbrace u_n v\,\vert\,u\in V^0, n\in\ZZ\rbrace$. This will follow from showing $V=\mathrm{span}\lbrace u_n v\,\vert\,u\in V, n\in\CC\rbrace$, because then compatibility of $Y$ with the $A$-grading shows
 \begin{align*}
  V  =\mathrm{span}\lbrace u_n v\,\vert u\in V, n\in\CC\rbrace = \bigoplus_{\beta\in A} \mathrm{span}\lbrace u_n v\,\vert\,u\in V^\beta, n\in\CC\rbrace\subseteq\bigoplus_{\beta\in A} V^{\alpha+\beta} = V,
 \end{align*}
so that $V^{\alpha+\beta}$ must equal $\mathrm{span}\lbrace u_n v\,\vert\, u\in V^\beta, n\in\CC\rbrace$ for all $\beta\in A$.

Now set $\widetilde{V}=\mathrm{span}\lbrace u_n v\,\vert\,u\in V, n\in\CC\rbrace$. To show that $\widetilde{V}=V$, note that since $v$ is non-zero and $A$-homogeneous, $\widetilde{V}$ is a non-zero $A$-graded subspace of $V$. Thus since $V$ is simple, we only need to show that $\widetilde{V}$ is closed under the the left action of the vertex operator $Y$. Specifically, we need to show that for all $v_1, v_2\in V$, the coefficients of the series $Y(v_1, x_1)Y(v_2, x_2)v$ are contained in $\widetilde{V}$.

The first thing to notice is that the easy generalization of \cite[Propositions 4.5.6 and 4.5.7]{LL} to intertwining operators among modules for the vertex operator algebra $V^0$ implies that for any $\beta\in A$, $\widetilde{V}\cap V^\beta$ is a $V^0$-submodule of $V^\beta$. This means that $\widetilde{V}$ is preserved by $L(0)$, and hence $\widetilde{V}$ is $L(0)$-graded.

Now consider the graded dual $V'=\bigoplus_{\alpha\in A, n\in\CC} (V^\alpha_{(n)})^*$, and let $\langle\cdot,\cdot\rangle$ denote the natural bilinear pairing between $V'$ and $V$. Define
\begin{equation*}
 \widetilde{V}^\perp =\lbrace v'\in V'\,\vert\,\langle v', \widetilde{v}\rangle =0\,\,\mathrm{for\,all}\,\,\widetilde{v}\in\widetilde{V}\rbrace.
\end{equation*}
Since $\widetilde{V}$ is doubly graded by $A$ and by conformal weight, the same is true of $\widetilde{V}^\perp$. Then we have $(\widetilde{V}^\perp)^\perp=\widetilde{V}$ because this relation holds when restricted to each finite-dimensional subspace $V^\alpha_{(n)}$ for $\alpha\in A$ and $n\in\CC$. Thus to show that the coefficients of $Y(v_1, x_1)Y(v_2,x_2)v$ are contained in $\widetilde{V}$ for any $v_1,v_2\in V$, it is sufficient to show that the series
\begin{equation*}
 \langle v', Y(v_1,x_1)Y(v_2,x_2)v\rangle = 0
\end{equation*}
for all $v'\in\widetilde{V}^\perp$.

By linearity, it is enough to consider the case that $v_1\in V^{\alpha_1}$ and $v_2\in V^{\alpha_2}$ for some $\alpha_1,\alpha_2\in A$. Then by \cite[Remark 12.31]{DL},
\begin{equation*}
 (x_1-x_2)^{\hat{b}(\alpha_1,\alpha_2)} x_1^{\hat{b}(\alpha_1,\alpha)} x_2^{\hat{b}(\alpha_2,\alpha)}\langle v', Y(v_1,x_1)Y(v_2,x_2)v\rangle
\end{equation*}
is the expansion (in non-negative powers of $x_2$) of a rational function $f(v'; x_1,x_2)$ on $\CC^2$ with possible poles only at $x_1=0$, $x_2=0$, and $x_1=x_2$. Moreover,
\begin{equation*}
 F(\alpha_1,\alpha_2,\alpha)x_0^{\hat{b}(\alpha_1,\alpha_2)} (x_2+x_0)^{\hat{b}(\alpha_1,\alpha)} x_2^{\hat{b}(\alpha_2,\alpha)}\langle v', Y(Y(v_1,x_0)v_2,x_2)v\rangle
\end{equation*}
is the expansion of $f(v'; x_2+x_0, x_2)$ in the direction of non-negative powers of $x_0$. But if $v'\in\widetilde{V}^\perp$, it is clear that the series expansion of $f(v'; x_2+x_0, x_2)$ equals zero, so the same is true of $f(v'; x_1, x_2)$. It follows that $\langle v', Y(v_1,x_1)Y(v_2,x_2)v\rangle=0$ whenever $v'\in\widetilde{V}^\perp$, completing the proof.
\end{proof}

Now we define automorphisms for an abelian intertwining algebra:
\begin{defi}
 An \textit{automorphism} of an abelian intertwining algebra $V$ associated to an abelian group $A$ with normalized $3$-cocycle $(F,\Omega)$ is a linear automorphism $g$ of $V$ which satisfies:
\begin{enumerate}
 \item For any $\alpha\in A$, $g\cdot V^\alpha=V^\alpha$.
 \item For any $v_1, v_2\in V$, $g\cdot Y(v_1,x)v_2=Y(g\cdot v_1, x)g\cdot v_2.$
 \item $g\cdot\vac=\vac$ and $g\cdot\omega =\omega$.
\end{enumerate}
\end{defi}

\begin{rema}
 We require an automorphism $g$ of $V$ to preserve the $A$-grading, and in fact $g$ preserves the conformal weight grading as well because $g\cdot\omega=\omega$ implies $g$ commutes with $L(0)$.
\end{rema}

\begin{exam}
 Suppose $\chi: A\rightarrow\CC^\times$ is a group homomorphism. Then compatibility of $Y$ with the $A$-gradation implies that $\chi$ determines an automorphism of any abelian intertwining algebra associated to $A$ via
 \begin{equation*}
  \chi\cdot v=\chi(\alpha) v
 \end{equation*}
for $v\in V^\alpha$, any $\alpha\in A$. In particular, viewing $A$ as discrete, the unitary dual $\widehat{A}$ of homomorphisms $A\rightarrow U(1)$ is a compact topological group acting faithfully as automorphisms of $V$.
\end{exam}

If $G$ is a group of automorphisms of an abelian intertwining algebra $V$, we use $V^G$ to denote the fixed points of $G$, that is, the subspace of all $v\in V$ such that $g\cdot v=v$ for all $g\in G$. It is clear that $V^G\cap V^0$ is a vertex operator subalgebra of $V^0$. Note also that $V^G\cap V^0=V^{\langle G, \widehat{A}\rangle}$. Since we will always want $V^G$ to be a vertex operator algebra, we will generally assume $\widehat{A}\subseteq G$. 
Note that since automorphisms of $V$ preserve the $A$-grading, $\widehat{A}$ will be central in any automorphism group $G$ that contains it.

\section{From \texorpdfstring{$G$}{G}-modules to \texorpdfstring{$V^G$}{VG}-modules}

In this section and the next, $V$ will be a simple abelian intertwining algebra and $G$ will be some compact group of automorphisms of $V$. Our goal in this section is to introduce a functor from $\rep G$ to the category of $V^G$-modules that are contained in $V$. The appropriate functor to use was constructed in \cite{KirillovOrbifoldI}; to show that it induces an equivalence of abelian categories, we need a Schur-Weyl-type decomposition of $V$ as a $G\times V^G$-module. Such a decomposition was obtained in \cite{DLM} in the case that $V$ is a vertex operator algebra, so we just need to show that essentially the same proof there carries over to the abelian intertwining alebra case. First, we establish the setting for this section and the next:
\begin{assum}\label{assum1}
We will always work under the following assumptions:
\begin{itemize}
 \item $A$ is a countable abelian group with normalized $3$-cocycle $(F,\Omega)$ on $A$ with values in $\mathbb{C}^\times$.
 
 \item $V$ is a simple abelian intertwining algebra associated to $A$ and $(F,\Omega)$.
 
 \item $G$ is a compact Lie group of automorphisms of $V$ containing $\widehat{A}$. (Note that $G$ could be finite if $A$ is.)
 
 \item The action of $G$ on $V$ is continuous in the sense that for any $\alpha\in A$, $n\in\CC$, the action of $G$ on the finite-dimensional vector space $V^\alpha_{(n)}$ is continuous with respect to the standard Euclidean topology on $V^\alpha_{(n)}$ (in particular, this topology may be given by any norm $\Vert \cdot\Vert$).
\end{itemize}
\end{assum}

\subsection{Schur-Weyl duality for abelian intertwining algebras}\label{sec:SW-duality}

In \cite{DLM}, it was proved that if $V$ is a simple vertex operator algebra and $G$ is a compact Lie group of automorphisms acting continuously on $V$, then as a $G\times V^G$ module,
\begin{equation*}
 V=\bigoplus_{\chi\in\widehat{G}} M_\chi\otimes V_\chi
\end{equation*}
where the sum runs over all irreducible finite-dimensional continuous characters of $G$, $M_\chi$ is the irreducible $G$-module corresponding to $\chi$, and the $V_\chi$ are (non-zero) distinct irreducible $V^G$-modules. Here, we show that the arguments in \cite{DLM} generalize to the case that $V$ is a suitable abelian intertwining algebra:

\begin{theo}
 Suppose $V$ is a simple abelian intertwining algebra associated to a countable abelian group $A$ with $3$-cocycle $(F,\Omega)$, and suppose $G$ is a compact Lie group of continuous automorphisms of $V$ containing $\widehat{A}$. Then as a $G\times V^G$-module,
 \begin{equation*}
  V=\bigoplus_{\chi\in\widehat{G}} M_\chi\otimes V_\chi,
 \end{equation*}
where the sum runs over all irreducible finite-dimensional characters of $G$, $M_\chi$ is the irreducible $G$-module corresponding to $\chi$, and the $V_\chi$ are non-zero distinct irreducible $V^G$-modules.
\end{theo}

\begin{proof}
In the abelian intertwining algebra setting, the decomposition of $V$ into a sum of tensor product modules works exactly as in \cite{DLM}: First, since $G$ acts continously on each finite-dimensional vector space $V^\alpha_{(n)}$, $V$ decomposes as the direct sum of finite-dimensional irreducible (continuous) $G$-modules. Thus for each $\chi\in\widehat{G}$, we set $V^\chi$ to be the sum of all $G$-submodules in $V$ that are isomorphic to $M_\chi$. So $V=\bigoplus_{\chi\in\widehat{G}} V^\chi$, although \textit{a priori} some $V^\chi$ could be zero.  Note that since $\widehat{A}\subseteq G$ and the $V^\alpha$ are inequivalent $\widehat{A}$-modules, for any $\chi$ we have $V^\chi\subseteq V^\alpha$ for some single $\alpha\in A$. Then, the $V^G$-submodule $V_\chi$ and the isomorphism $V^\chi\cong M_\chi\otimes V_\chi$ are obtained by taking $V_\chi$ to be the space of ``highest weight vectors'' in $V^\chi$ with respect to a certain (non-canonical) upper-triangular subalgebra of the group algebra $L^1(G)$. 

We still need to check that the proofs of the following assertions from \cite{DLM} carry through in our setting:
\begin{enumerate}
 \item $V^\chi$ is non-zero for any $\chi\in\widehat{G}$. Here, the proof in \cite{DLM} relies on the following facts:
 \begin{enumerate}
 \item There is a finite-dimensional faithful $G$-submodule contained in $V$.
 \item If $V^\chi$ and $V^\psi$ are non-zero, so is $V^\rho$ whenever $M_\rho$ is a submodule of $M_\chi\otimes M_\psi$.
 \item If $V^\chi\neq 0$, so is $V^{\chi^*}$ where $\chi^*$ is the dual character to $\chi$.
\end{enumerate}

\item The $V_\chi$ are distinct irreducible $V^G$-modules.
\end{enumerate}

The proof of 1(a) in \cite{DLM} uses only the compactness and finite-dimensionality of the Lie group $G$, as well as on the existence of an increasing sequence of finite-dimensional $G$-invariant subspaces of $V$ that exhaust $V$. Unlike in \cite{DLM}, in the abelian intertwining algebra setting we cannot necessarily take the sequence of subspaces  $\lbrace V_N\rbrace$ where $V_N=\bigoplus_{n\leq N} V_{(n)}$, since this subspace might not be finite dimensional. However, since we are assuming that $A$ is countable, we can enumerate $A$ as $\lbrace\alpha_0, \alpha_1,\alpha_2,\ldots\rbrace$ and then take
\begin{equation*}
 V_N =\bigoplus_{i=0}^N \bigoplus_{\mathrm{Re}\,n\leq N} V^{\alpha_i}_{(n)},
\end{equation*}
which does give a sequence of finite-dimensional $G$-submodules of $V$ that exhausts $V$, by our assumptions on the grading of an abelian intertwining algebra. The proof of point 1(c) in \cite{DLM} goes through unchanged except that we need to use the generalization of \cite[Proposition 2.1]{DLM}, obtained here in the proof of Proposition \ref{Valphasimple}, that $V=\mathrm{span}\lbrace u_n v\,\vert\,u\in V, n\in\CC\rbrace$ for any $A$-homogeneous $v\in V$.

Now the proof of point 1(b) in \cite{DLM} relies on \cite[Lemma 3.1]{DM1}. Here we state this lemma in the somewhat generalized form needed for the abelian intertwining algebra setting:
\begin{lemma}
If $\lbrace v_1^{(i)}\rbrace_{i=1}^n$ is a set of non-zero vectors in $V^{\alpha_1}$ and $\lbrace v_2^{(i)}\rbrace_{i=1}^n$ is a set of linearly independent vectors in $V^{\alpha_2}$ for some $\alpha_1,\alpha_2\in A$, then
\begin{equation*}
 \sum_{i=1}^n Y(v_1^{(i)}, x)v_2^{(i)} \neq 0.
\end{equation*}
\end{lemma}
In light of Proposition \ref{Valphasimple}, this lemma (as well as the original \cite[Lemma 3.1]{DM1}), is a corollary of the following:
\begin{lemma}\label{DMgen}
 Suppose $W_1$ and $W_2$ are irreducible modules for a vertex operator algebra algebra $V$, $W_3$ is a third $V$-module, and $\mathcal{Y}$ is an intertwining operator of type $\binom{W_3}{W_1\,W_2}$. If $\lbrace w_1^{(i)}\rbrace_{i=1}^n$ is a set of non-zero vectors in $W_1$ and $\lbrace w_2^{(i)}\rbrace_{i=1}^n$ is a set of linearly independent vectors in $W_2$, then
 \begin{equation*}
  \sum_{i=1}^n \mathcal{Y}(w_1^{(i)}, x)w_2^{(i)} \neq 0.
 \end{equation*}
\end{lemma}
The proof of this lemma is the same as the proof of \cite[Lemma 3.1]{DM1}: one uses commutativity for the intertwining operator $\cY$ to show that if $\sum_{i=1}^n \mathcal{Y}(w_1^{(i)}, x)w_2^{(i)} =0$, then also $\sum_{i=1}^n \mathcal{Y}(w_1^{(i)}, x)aw_2^{(i)}=0$, where $a$ is any linear combination of products of the form $v^{(1)}_{n_1}\cdots v^{(k)}_{n_k}$ for $v^{(1)},\ldots, v^{(k)}\in V$ and $n_1,\ldots,n_k\in\ZZ$. Then a density theorem argument using the irreducibility of $W_2$ and the linear independence of the $w_2^{(i)}$ shows that $a$ can be chosen so that $a w^{(i)}_2 =\delta_{i,j} w^{(i)}_2$ for any particular $1\leq j\leq n$. Thus, $\mathcal{Y}(w_1^{(j)}, x)w^{(j)}_2 =0$, and \cite[Proposition 11.9]{DL} implies $w_1^{(j)}$ must be zero for all $j$. This completes the verification that the proof in \cite{DLM} extends to our abelian intertwining algebra setting to show that $V^\chi\neq 0$ for every $\chi\in\widehat{G}$.

Now to verify that the $V_\chi$ remain distinct irreducible modules in the abelian intertwining algebra setting: the proof in \cite{DLM} is based on Howe's theory of dual pairs (see for instance \cite{Ho}) and essentially requires components of vertex operators to fill up the endomorphism ring of finite-dimensional conformal-weight-graded $G$-submodules of $V$. To apply the proof in the abelian intertwining algebra setting, we need to verify the following assertion:

\begin{itemize} \item Suppose $f$ is a $G$-endomorphism of the finite-dimensional $G$-module  $W$ where $W$ is either $\bigoplus_{\mathrm{Re}\,n\leq N} V^\alpha_{(n)}$ for some $\alpha\in A$ or $\bigoplus_{\mathrm{Re}\,n\leq N} (V^{\alpha_1}_{(n)}\oplus V^{\alpha_2}_{(n)})$ for $\alpha_1,\alpha_2\in A$, where $N$ is some fixed integer. Then for any $n\in\CC$ with $\mathrm{Re}\,n\leq N$,
\begin{equation*}
 f\vert_{W_{(n)}} =\sum_i v^{(i)}_{n_i}\vert_{W_{(n)}}
\end{equation*}
for some $v^{(i)}\in V^G$ and $n_i\in\ZZ$ (which could depend on $n$ as well as on $f$).
\end{itemize}

It is enough to verify the above assertion for a $G$-endomorphism $f$ of $W=\bigoplus_{\mathrm{Re}\,n\leq N} (V^{\alpha_1}_{(n)}\oplus V^{\alpha_2}_{(n)})$. The argument is basically the same as for \cite[Lemma 2.2]{DLM}, but we provide a few more details.  For $m\in\ZZ$ let $P_{\leq m}$ denote the projection from $W$ to $\bigoplus_{\mathrm{Re}\,n\leq m} W_{(n)}$; this is a $G$-module endomorphism of $W$ because $G$ preserves the conformal weight spaces of $W$. Also let $P_W$ denote the projection from $V$ onto $W$ with respect to the conformal weight and $A$-gradings, also a $G$-module homomorphism.

Now we claim that
\begin{equation*}
 E=\mathrm{span}\lbrace P_W v_n P_{\leq m}\,\vert\,v\in V^0, n,m\in\ZZ\rbrace
\end{equation*}
is a subalgebra of $\Endo W$. To verify this claim, first note that $E$ contains $1_W$ because this equals $P_W \vac_{-1} P_{\leq N}$. Next, we need to show that
\begin{equation*}
 (P_W v_n P_{\leq m})(P_W v'_{n'} P_{\leq m'})\in E 
\end{equation*}
for $v, v'\in V^0$, $m,n,m',n'\in\ZZ$. We may assume that $v, v'$ are homogeneous with respect to the conformal weight gradation of $V^0$. In this case, for homogeneous $w\in W$,
\begin{align*}
 P_{\leq m} P_W v'_{n'} w & =\left\lbrace\begin{array}{cl}
                                        v'_{n'} w  & \mathrm{if}\,\, \mathrm{Re}\,\mathrm{wt}\,w\leq \mathrm{min}(N+n'+1-\mathrm{wt}\,v', m+n'+1-\mathrm{wt}\,v') \\
                                        0 & \mathrm{otherwise} \\
                                       \end{array}\right. \nonumber\\
                                       & = v'_{n'} P_{\leq M(v', n', m)} w
\end{align*}
where $M(v',n',m)=\mathrm{min}(N+n'+1-\mathrm{wt}\,v', m+n'+1-\mathrm{wt}\,v')$. Thus
\begin{equation*}
 (P_W v_n P_{\leq m})(P_W v'_{n'} P_{m'})=P_W v_n v'_{n'} P_{\leq\mathrm{min}(m',M(v',n',m))}.
\end{equation*}
Now by \cite[Proposition 4.5.7]{LL}, we have, for any $w\in V$,
\begin{equation}\label{algproof}
 v_n v'_{n'} w =\sum_{i=0}^{M}\sum_{j=0}^L \binom{n-L}{i}\binom{L}{j} (v_{n-L-i+j} v')_{n'+L+i-j} w
\end{equation}
where $L$ and $M$ are nonnegative integers chosen such that $v_k w =0$ for $k\geq L$ and $v'_k w =0$ for $k>M+n'$. Now if we choose $L$ and $M$ large enough so that these relations hold when $w$ is any vector of a (finite) basis for $W$, then \eqref{algproof} holds for any $w\in W$. Thus we get
\begin{equation*}
 (P_W v_n P_{\leq m})(P_W v'_{n'} P_{\leq m'}) =\sum_{i=0}^M\sum_{j=0}^L \binom{n-L}{i}\binom{L}{j} P_W(v_{n-L-i+j} v')_{n'+L+i-j} P_{\leq\mathrm{min}(m',M(v',n',m))}\in E,
\end{equation*}
completing the proof that $E$ is an algebra in $\Endo W$.

Now we observe that as an $E$-module, $W=(W\cap V^{\alpha_1})\oplus(W\cap V^{\alpha_2})$, and by applying \cite[Proposition 4.5.6]{LL} and the fact that each $V^\alpha$ is an irreducible $V^0$-module, we see that these two summands of $W$ are irreducible $E$-modules. Since $W$ is a completely reducible $E$-module, a density theorem argument shows that $E=\Endo (W\cap V^{\alpha_1})\oplus\Endo (W\cap V^{\alpha_2})$. Since $G$ preserves the $A$-grading of $V$ as well as $W$, $G$ acts on this subalgebra of $\Endo W$ by conjugation, so $E$ decomposes as a direct sum of irreducible $G$-modules. 

For any irreducible character $\chi$ of $G$, the sum of all $G$-submodules in $E$ isomorphic to $M_\chi$ is
\begin{equation*}
E^\chi=\mathrm{span}\lbrace P_W v_n P_{\leq m}\,\vert\, v\in V^0\cap V^\chi, m,n\in\ZZ\rbrace;
\end{equation*}
this is because $g v_n g^{-1}=(g\cdot v)_n$ and because $P_W$ and $P_{\leq m}$ are $G$-module homomorphisms. In particular, if $f$ is an $A$-grading-preserving $G$-endomorphism of $W$, $f$ is a linear combination
\begin{equation*}
 f=\sum_i P_W v^{(i)}_{n_i} P_{\leq m_i}
\end{equation*}
where $v^{(i)}\in V^G$ are homogeneous and $m_i,n_i\in\ZZ$. So if $\mathrm{Re}\,n\leq N$, we have
\begin{equation*}
 f\vert_{W_{(n)}}={\sum_i}' v^{(i)}_{n_i}\vert_{W_{(n)}},
\end{equation*}
where the sum is restricted to $i$ such that $m_i\geq \mathrm{Re}\,n$ and $\mathrm{wt}\,v^{(i)}+\mathrm{Re}\,n-n_i-1\leq N$. Note that if $f\vert_{W_{(n)}}\neq 0$, this restricted sum cannot be empty. Finally, we simply note that any $G$-endomorphism $f$ of $W$ automatically preserves the $A$-grading because such an endomorphism commutes with $\widehat{A}\subseteq G$.
\end{proof}

\subsection{The functor}\label{sec:Phi}

We continue to work in the setting of Assumption \ref{assum1}. Here, we discuss a functor $\Phi$ introduced in \cite{KirillovOrbifoldI} from $\rep G$ to the (semisimple) abelian category $\cC_V$ of $V^G$-modules generated by the $V_\chi$, $\chi\in\widehat{G}$, appearing in the decomposition of $V$ as a $G\times V^G$-module. In \cite[Theorem 2.11]{KirillovOrbifoldI}, it was shown that when $G$ is finite and $\cC_V$ is contained in a rigid, semisimple, braided tensor category of $V^G$-modules, $\Phi$ is a braided tensor equivalence. However, rigidity is generally rather difficult to establish for a category of vertex operator algebra modules. On the other hand, if $G$ is finite abelian, $V$ is a vertex operator algebra, and the $V_\chi$ are contained in a braided tensor category of $V^G$-modules, the rigidity of the modules $V_\chi$ was proven in \cite{Miy}, \cite{CarM} (and see \cite[Appendix A]{CKLR} for the compact abelian case). We will settle the question of when $\Phi$ is a braided tensor equivalence in the next section, where we show that for general compact $G$ and an abelian intertwining algebra $V$, $\Phi$ is a braided tensor equivalence from $\rep_{A,F,\Omega} G$ to $\cC_V$ exactly when $\cC_V$ is contained in some vertex (and thus also braided) tensor category of $V^G$-modules.

We now recall the functor $\Phi: \rep\,G\rightarrow\cC_V$ constructed in \cite{KirillovOrbifoldI}: Given a finite-dimensional continuous $G$-module $M$, $M\otimes V$ is a $V^G$-module (not in $\cC_V$ unless $G$ is finite) with vertex operator $1_M\otimes Y$, as well as a $G$-module with the tensor product $G$-module structure. Then the $G$-fixed points of $M\otimes V$ is another $V^G$-module. Note that since as a $G\times V^G$-module
\begin{equation*}
 M\otimes V=\bigoplus_{\chi\in\widehat{G}} (M\otimes M_\chi)\otimes V_\chi,
\end{equation*}
$(M\otimes V)^G$ contains (finitely many) copies of $V_\chi$ only if $M_\chi^*$ is contained in $M$. Thus $(M\otimes V)^G$ is an object of $\cC_V$. Thus we can define the functor $\Phi$ by:
\begin{itemize}
 \item For any module $M$ in $\rep\,G$, $\Phi(M)=(M\otimes V)^G$.
 
 \item For any morphism $f: M_1\rightarrow M_2$ of finite-dimensional $G$-modules, $\Phi(f)= (f\otimes 1_V)\vert_{(M_1\otimes V)^G}$.
\end{itemize}
Note that the image of $\Phi(f)$ indeed lies in the $G$-fixed points of $M_2\otimes V$, and that $\Phi(f)$ is also a homomorphism of $V^G$-modules.

\begin{propo}\label{PhiAbEquiv}
 For each irreducible character $\chi\in\widehat{G}$, there is a $V^G$-module isomorphism $\varphi_\chi: V_{\chi}\rightarrow\Phi(M_\chi^*)$ given by
 \begin{equation*}
  \varphi_\chi(v_\chi)=\sum_i m_{\chi, i}'\otimes(m_{\chi, i}\otimes v_\chi)
 \end{equation*}
for $v_\chi\in V_\chi$, where $\lbrace m_{\chi, i}\rbrace$ is a basis for $M_\chi$ and $\lbrace m_{\chi, i}'\rbrace$ is the corresponding dual basis of $M_\chi^*$. Moreover, $\Phi$ is an equivalence of abelian categories between $\rep\,G$ and $\cC_V$.
\end{propo}
\begin{proof}
 It is evident that the indicated linear map $\varphi_\chi$ is an injective $V^G$-module homomorphism from $V_\chi$ to $(M_\chi^*\otimes V)^G$. It is also surjective because $M_\chi^*\otimes M_\psi$ contains a (necessarily one-dimensional) non-zero $G$-invariant subspace if and only if $\psi=\chi$.
 
 For the equivalence of abelian categories, we have just shown that every irreducible object $V_\chi$ of $\cC_V$ is isomorphic to some $\Phi(M)$ where $M=M_\chi^*$ is irreducible. Also, since the $V_\chi$ are irreducible and distinct, $\Phi$ induces isomorphisms on spaces of morphisms between irreducible modules. Now it follows that $\Phi$ is an equivalence because both $\rep G$ and $\cC_V$ are semisimple.
\end{proof}

To address the question of whether $\Phi$ gives an equivalence of tensor categories, we will need the natural transformation $J$ of \cite{KirillovOrbifoldI} to relate tensor products in $\cC_V$ to tensor products in $\rep_{A,F,\Omega}\,G$. In fact, even if the $V_\chi$ are not necessarily contained in a braided tensor category of $V^G$-modules, we can express $J$ as a linear map between spaces of intertwining operators. Given modules $W_1$, $W_2$, and $W_3$ for a vertex operator algebra, we use $\mathcal{V}^{W_3}_{W_1\,W_2}$ to denote the vector space of intertwining operators of type $\binom{W_3}{W_1\,W_2}$. 

For finite-dimensional continuous $G$-modules $M_1$, $M_2$, and $M_3$, we define
\begin{equation}\label{Jdef}
 J^{M_3}_{M_1, M_2}: \hom_G(M_1\otimes M_2, M_3)\rightarrow\mathcal{V}^{\Phi(M_3)}_{\Phi(M_1)\,\Phi(M_2)}
\end{equation}
as follows: Given a $G$-module homomorphism $f: M_1\otimes M_2\rightarrow M_3$, we have an intertwining operator $\mathcal{Y}_f$ of type $\binom{M_3\otimes V}{M_1\otimes V\,M_2\otimes V}$ given by
\begin{equation*}
 \cY_f(m_1\otimes v_1, x)(m_2\otimes v_2) = f(m_1\otimes m_2)\otimes Y(v_1,x)v_2
\end{equation*}
for $m_1\in M_1$, $m_2\in M_2$, and $v_1,v_2\in V$. Since $f$ and the components of $Y$ are $G$-module homomorphisms, the restriction of $\mathcal{Y}_f$ to $G$-fixed points of $(M_1\otimes V)\otimes(M_2\otimes V)$ maps into $G$-fixed points of $(M_3\otimes V)((x))$. Hence
\begin{equation*}
 J^{M_3}_{M_1,M_2}(f)=\mathcal{Y}_f\vert_{\Phi(M_1)\otimes\Phi(M_2)}
\end{equation*}
is a $V^G$-intertwining operator of type $\binom{\Phi(M_3)}{\Phi(M_1)\,\Phi(M_2)}$.

We do not need tensor category structure on $\cC_V$ or any larger category of $V^G$-modules to prove that $J^{M_3}_{M_1, M_2}$ is injective:
\begin{propo}\label{Jsurjective}
 For any $G$-modules $M_1$, $M_2$, and $M_3$ in $\rep\,G$, the linear map $J^{M_3}_{M_1, M_2}$ is injective.
\end{propo}
\begin{proof}
Because $\rep\,G$ is semisimple, we may assume that $M_1=M_{\chi}^*$, $M_2=M_\psi^*$, and $M_3=M_\rho^*$ for irreducible characters $\chi,\psi,\rho\in\widehat{G}$. We shall show that if the intertwining operator $\mathcal{Y}=J^{M_3}_{M_1,M_2}(f)\circ(\varphi_\chi\otimes\varphi_\psi)$ of type $\binom{\Phi(M_\rho^*)}{V_\chi\,V_\psi}$ equals zero, then $f=0$ as well. Thus suppose for all $v_\chi\in V_\chi$ and $v_\psi\in V_\psi$ we have
\begin{align*}
 0 = \cY(v_\chi, x)v_\psi & =\sum_{i, j}\cY_f\left(m_{\chi, i}'\otimes(m_{\chi, i}\otimes v_\chi), x\right)\left(m_{\psi, j}'\otimes(m_{\psi,j}\otimes v_\psi)\right)\nonumber\\
 & =\sum_{i,j} f(m_{\chi, i}'\otimes m_{\psi,j}')\otimes Y(m_{\chi,i}\otimes v_\chi, x)(m_{\psi,j}\otimes v_\psi)\nonumber\\
 & =\sum_{i,j,k} \langle f(m_{\chi,i}'\otimes m_{\psi,j}'), m_{\rho,k}\rangle m_{\rho,k}'\otimes Y(m_{\chi,i}\otimes v_\chi, x)(m_{\psi,j}\otimes v_\psi)\nonumber\\
 &=\sum_k\left(m_{\rho,k}'\otimes\sum_j Y\left(\sum_i \langle f(m_{\chi,i}'\otimes m_{\psi,j}'), m_{\rho,k}\rangle(m_{\chi,i}\otimes v_\chi), x)\right)(m_{\psi,j}\otimes v_\psi)\right).
\end{align*}
Since the basis vectors $m_{\rho,k}'$ are linearly independent, we have
\begin{equation*}
 \sum_j Y\left(\sum_i \langle f(m_{\chi,i}'\otimes m_{\psi,j}'), m_{\rho,k}\rangle(m_{\chi,i}\otimes v_\chi), x)\right)(m_{\psi,j}\otimes v_\psi) = 0
\end{equation*}
for all $k$. Now we may take $v_\psi$ to be non-zero so that the vectors $\lbrace m_{\psi,j}\otimes v_\psi\rbrace$ are linearly independent. Then because $V$ is simple, \cite[Lemma 3.1]{DM1} (or more precisely, its generalization Lemma \ref{DMgen} here) implies that
\begin{equation*}
 \sum_i \langle f(m_{\chi,i}'\otimes m_{\psi,j}'), m_{\rho,k}\rangle(m_{\chi,i}\otimes v_\chi) = 0
\end{equation*}
for all $j, k$. Assuming also that $v_\chi$ is non-zero so that the vectors $\lbrace m_{\chi,i}\otimes v_\chi\rbrace$ are linearly independent, this means that
\begin{equation*}
 \langle f(m_{\chi,i}'\otimes m_{\psi,j}'), m_{\rho,k}\rangle = 0
\end{equation*}
for all $i,j,k$, that is, $f=0$.
\end{proof}

\section{The main theorems}\label{sec:maintheorems}

Continuing to work under Assumption \ref{assum1}, our goal in this section is to show that if the linear maps $J^{M_3}_{M_1, M_2}: \hom_G(M_1\otimes M_2, M_3)\rightarrow\mathcal{V}^{\Phi(M_3)}_{\Phi(M_1)\,\Phi(M_2)}$ are surjective as well as injective, then the category $\mathcal{C}_V$ of $V^G$-modules has vertex tensor category structure. Conversely, if we know that $\mathcal{C}_V$ is contained in any vertex tensor category of $V^G$-modules (which could be larger than $\cC_V$), then $\Phi$ induces a braided tensor equivalence (and in particular the maps $J^{M_3}_{M_1, M_2}$ are surjective). In this second case, the tensor products of modules in $\cC_V$ are thus independent of the larger category $\cC$ (recall Remark \ref{tensprodsindiffC}), and since $\rep_{A,F,\Omega}\,G$ is rigid (and indeed ribbon), $\cC_V$ will be a ribbon braided tensor subcategory of $V^G$-modules.

\subsection{Equal fusion rules implies vertex tensor category structure}\label{subsec:firstmaintheo}

In this subsection, we assume that the linear maps $J^{M_3}_{M_1\,M_2}$ are isomorphisms, and in particular, surjective. This allows us to identify the $P(z)$-tensor products in the category $\cC_V$ of $V^G$-modules. Moreover, this effectively means that all $V^G$-intertwining operators in $\cC_V$ are derived from the vertex operator on $V$, which has the associativity properties needed to establish the $P(z_1,z_2)$-associativity isomorphisms in $\cC_V$.

\begin{theo}\label{fusionruletovrtxtens}
 If the linear map $J^{M_3}_{M_1, M_2}: \hom_G(M_1\otimes M_2, M_3)\rightarrow\mathcal{V}^{\Phi(M_3)}_{\Phi(M_1)\,\Phi(M_2)}$ of \eqref{Jdef} is surjective for all finite-dimensional continuous $G$-modules $M_1$, $M_2$, $M_3$, then the category $\cC_V$ of $V^G$-modules has vertex tensor category structure as described in Section \ref{subsec:vrtxtenscat}.
\end{theo}
\begin{proof}
 We first need to show the existence of $P(z)$-tensor products in $\cC_V$ for $z\in\CC^\times$; then the existence of parallel transport, unit, and braiding isomorphisms follow immediately from the universal property of the $P(z)$-tensor product (see \cite[Section 12.2]{HLZ8}). Then we need to establish existence of
 $P(z_1,z_2)$-associativity isomorphisms in $\cC_V$ for $\vert z_1\vert>\vert z_2\vert>\vert z_1-z_2\vert>0$. Finally, the only extra condition we need to verify for the coherence properties of \cite[Theorem 12.15]{HLZ8} is suitable convergence of products of three intertwining operators in $\cC_V$ (see \cite[Assumption 12.2]{HLZ8}).
 
 For $z\in\CC^\times$, the $P(z)$-tensor product of modules $W_1$, $W_2$ in $\cC_V$ can be constructed essentially as in \cite[Proposition 4.33]{HLZ3}. Specifically, since $\Phi$ is an equivalence of categories, we may take $W_1=\Phi(M_1)$ and $W_2=\Phi(M_2)$ for $M_1$, $M_2$ modules in $\rep G$. Then we can take
 \begin{equation*}
  \Phi(M_1)\tens_{P(z)} \Phi(M_2) = \bigoplus_{\chi\in\widehat{G}} \hom_G(M_1\otimes M_2, M_\chi)^*\otimes\Phi(M_\chi)
 \end{equation*}
with tensor product $P(z)$-intertwining map
\begin{equation*}
 \boxtimes_{P(z)}=\sum_{\chi\in \widehat{G}}\sum_i f_i'\otimes \left(J^{M_\chi}_{M_1, M_2}(f_i)\right)(\cdot,e^{\log z})\cdot,
\end{equation*}
where for each $\chi\in\widehat{G}$, $\lbrace f_i\rbrace$ represents any basis of $\hom_G(M_1\otimes M_2,M_\chi)$, and $\lbrace f_i'\rbrace$ is the dual basis. Note that although $\widehat{G}$ may be infinite, $\hom_G(M_1\otimes M_2, M_\chi)$ is non-zero for only finitely many $\chi$, so $\Phi(M_1)\tens_{P(z)}\Phi(M_2)$ is an object of $\cC_V$.

Now to show that $\Phi(M_1)\tens_{P(z)}\Phi(M_2)$ satisfies the universal property of a $P(z)$-tensor product, consider any $P(z)$-intertwining map $I$ of type $\binom{W_3}{\Phi(M_1)\,\Phi(M_2)}$ where $W_3$ is a module in $\cC_V$. Since $\cC_V$ is semisimple and equivalent as a category to $\rep G$, we may take $W_3=\Phi(M_\psi)$ for some $\psi\in\widehat{G}$. Then because $J^{M_\psi}_{M_1, M_2}$ is surjective by assumption, $\cY_I=J^{M_\psi}_{M_1, M_2}(f)$ for some $G$-module homomorphism $f: M_1\otimes M_2\rightarrow M_\psi$. So one needs to verify that 
\begin{equation}\label{univpropcond}
 J^{M_\psi}_{M_1, M_2}(f)=\eta\circ\left(\sum_{\chi, i} f_i'\otimes J^{M_\chi}_{M_1, M_2}(f_i)\right)
\end{equation}
for some unique $V^G$-module homomorphism $\eta: \Phi(M_1)\tens_{P(z)}\Phi(M_2)\rightarrow\Phi(M_\psi)$. In fact, the construction of $\Phi(M_1)\tens_{P(z)}\Phi(M_2)$ shows that $\eta$ amounts to a linear functional on $\hom_G(M_1\otimes M_2, M_\psi)^*$, that is, a homomorphism in $\hom_G(M_1\otimes M_2, M_\psi)$. Naturally enough, one then uniquely takes $\eta=f$ for \eqref{univpropcond} to hold. This establishes the existence of $P(z)$-tensor products in $\cC_V$.

Now to establish the existence of $P(z_1,z_2)$-associativity isomorphisms, it is sufficient to prove the following convergence and associativity of intertwining operators in $\cC_V$ (see for instance \cite[Theorem 9.29, Corollary 9.30]{HLZ6}):
\begin{itemize}
 \item Given modules $W_1$, $W_2$, $W_3$, $W_4$, and $X_1$ in $\cC_V$ and intertwining operators $\cY_1$, $\cY_2$ of types $\binom{W_4}{W_1\,X_1}$ and $\binom{X_1}{W_2\,W_3}$, respectively, the series
 \begin{equation*}
  \langle w_4', \cY_1(w_1,z_1)\cY_2(w_2,z_2)w_3\rangle =\sum_{h\in\CC} \langle w_4', \cY_1(w_1,z_1)\pi_h(\cY_2(w_2,z_2)w_3)\rangle
 \end{equation*}
converges absolutely (to a multivalued analytic function) when $\vert z_1\vert>\vert z_2\vert>0$ for any $w_1\in W_1$, $w_2\in W_2$, $w_3\in W_3$, and $w_4'\in W_4'$. Moreover, there is a module $X_2$ in $\cC_V$ and intertwining operators $\cY^1$, $\cY^2$ of types $\binom{W_4}{X_2\,W_3}$ and $\binom{X_2}{W_1\,W_2}$, respectively, such that
\begin{equation*}
  \langle w_4', \cY^1(\cY^2(w_1,z_0)w_2,z_2)w_3\rangle =\sum_{h\in\CC} \langle w_4', \cY^1(\pi_h(\cY^2(w_1,z_0)w_2),z_2)w_3\rangle
\end{equation*}
converges absolutely when $\vert z_2\vert>\vert z_0\vert>0$ for any $w_1\in W_1$, $w_2\in W_2$, $w_3\in W_3$, and $w_4'\in W_4'$, and such that
\begin{equation*}
 \langle w_4', \cY_1(w_1,z_1)\cY_2(w_2,z_2)w_3\rangle = \langle w_4', \cY^1(\cY^2(w_1,z_1-z_2)w_2,z_2)w_3\rangle
\end{equation*}
(as multivalued functions) in the region $\vert z_1\vert>\vert z_2\vert>\vert z_1-z_2\vert>0$.

\item Conversely, given modules $W_1$, $W_2$, $W_3$, $W_4$, $X_2$ and intertwining operators $\cY^1$, $\cY^2$, as above, there is a module $X_1$ in $\cC_V$ and intertwining operators $\cY_1$, $\cY_2$ as above such that
\begin{equation*}
 \langle w_4', \cY^1(\cY^2(w_1,z_0)w_2,z_2)w_3\rangle = \langle w_4', \cY_1(w_1,z_0+z_2)\cY_2(w_2,z_2)w_3\rangle
\end{equation*}
in the region $\vert z_0+z_2\vert>\vert z_2\vert>\vert z_0\vert>0$.
\end{itemize}

To establish this convergence and associativity of intertwining operators in $\cC_V$, we note that because $\cC_V$ is semisimple, all multivalued functions obtained from compositions of intertwining operators as above are sums of ones obtained by taking $W_i$ to be simple $V^G$-modules of the form $V_\chi$ for $\chi\in\widehat{G}$. Moreover, because $\Phi$ is a categorical equivalence, we may freely use either $V_\chi$ or $\Phi(M_\chi^*)$ together with the isomorphisms $\varphi_\chi: V_\chi\rightarrow\Phi(M_\chi^*)$ of the previous section. Thus we consider a product of intertwining operators
\begin{equation*}
 \langle w_4', \cY_1(w_1,z_1)\cY_2(w_2,z_2)w_3\rangle
\end{equation*}
where $\cY_1$ is of type $\binom{\Phi(M_4)}{V_{\chi_1}\,\Phi(N_1)}$ and $\cY_2$ is of type $\binom{\Phi(N_1)}{V_{\chi_2}\,V_{\chi_3}}$, for $\chi_1,\chi_2,\chi_3\in\widehat{G}$ and $M_4$, $N_1$ irreducible $G$-modules. Since by assumption the linear maps $J^{M_3}_{M_1, M_2}$ are surjective, we can take
\begin{equation*}
 \cY_1=J^{M_{4}}_{M_{\chi_1}^*, N_1}(f_1)\circ(\varphi_{\chi_1}\otimes 1_{\Phi(N_1)})
\end{equation*}
and
\begin{equation*}
 \cY_2=J^{N_1}_{M_{\chi_2}^*, M_{\chi_3}^*}(f_2)\circ(\varphi_{\chi_2}\otimes\varphi_{\chi_3})
\end{equation*}
for $f_1\in\hom_G(M_{\chi_1}^*\otimes N_1, M_4)$ and $f_2\in\hom_G(M_{\chi_2}^*\otimes M_{\chi_3}^*, N_1)$.

Then using $\lbrace m_{\chi_1, i}\rbrace$, $\lbrace m_{\chi_2, j}\rbrace$, and $\lbrace m_{\chi_3, k}\rbrace$ to denote bases of $M_{\chi_1}$, $M_{\chi_2}$, and $M_{\chi_3}$, and using $\lbrace m_{\chi_1, i}'\rbrace$, $\lbrace m_{\chi_2, j}'\rbrace$, and $\lbrace m_{\chi_3, k}'\rbrace$ to denote their respective dual bases, we have from the definitions
\begin{align*}
 \langle w_4', \cY_1 & (w_1,z_1)\cY_2(w_2,z_2)w_3\rangle = \sum_{j,k} \langle w_4', \cY_1(w_1,z_1)\left(f_2(m_{\chi_2,j}'\otimes m_{\chi_3, k}')\otimes Y(m_{\chi_2,j}\otimes w_2, z_2)(m_{\chi_3,k}\otimes w_3)\right)\rangle\nonumber\\
 & =\sum_{i,j,k}\left\langle w_4', f_1(m_{\chi_1,i}'\otimes f_2(m_{\chi_2,j}'\otimes m_{\chi_3,k}'))\otimes Y(m_{\chi_1,i}\otimes w_1, z_1)Y(m_{\chi_2,j}\otimes w_2, z_2)(m_{\chi_3,k}\otimes w_3)\right\rangle,
\end{align*}
and this series converges absolutely when $\vert z_1\vert>\vert z_2\vert>0$. If we write
\begin{equation*}
 g=f_1\circ(1_{M_{\chi_1}^*}\otimes f_2)\circ\cA_{M_{\chi_1}^*, M_{\chi_2}^*, M_{\chi_3}^*}^{-1},
\end{equation*}
and apply Proposition \ref{abintwalgassoc} and Remark \ref{abintwalgassocreal} in the case that $z_1$, $z_2$ are positive real numbers such that $z_1>z_2>z_1-z_2>0$, we get
\begin{align*}
 \langle w_4', & \cY_1 (w_1,e^{\log z_1})\cY_2(w_2,e^{\log z_2})w_3\rangle = F(\alpha_1,\alpha_2,\alpha_3)\cdot\nonumber\\
 &\hspace{2em} \cdot\sum_{i,j,k} \langle g((m_{\chi_1,i}'\otimes m_{\chi_2,j}')\otimes m_{\chi_3, k}')\otimes Y(Y(m_{\chi_1,i}\otimes w_1, e^{\log(z_1-z_2)})(m_{\chi_2,j}\otimes w_2), e^{\log z_2})(m_{\chi_3,k}\otimes w_3)\rangle\nonumber\\
 & = \left\langle w_4', \cY^1\left(\sum_{i,j} (m_{\chi_1,i}'\otimes m_{\chi_2,j}')\otimes Y(m_{\chi_1,i}\otimes w_1, e^{\log(z_1-z_2)})(m_{\chi_2,j}\otimes w_2), e^{\log z_2}\right)w_3\right\rangle\nonumber\\
 & = \langle w_4', \cY^1(\cY^2(w_1, e^{\log(z_1-z_2)})w_2,e^{\log z_2})w_3\rangle,
\end{align*}
where $V^{\chi_i}\subseteq V^{\alpha_i}$ for $i=1,2,3$,
\begin{equation*}
 \cY^1=F(\alpha_1,\alpha_2,\alpha_3)J^{M_4}_{M_{\chi_1}^*\otimes M_{\chi_2}^*, M_{\chi_3}^*}(g)\circ(1_{M_{\chi_1}^*\otimes M_{\chi_2}^*}\otimes\varphi_{\chi_3})
\end{equation*}
is an intertwining operator of type $\binom{\Phi(M_4)}{\Phi(M_{\chi_1}^*\otimes M_{\chi_2}^*)\,V_{\chi_3}}$ and
\begin{equation*}
 \cY^2=J^{M_{\chi_1}^*\otimes M_{\chi_2}^*}_{M_{\chi_1}^*, M_{\chi_2}^*}(1_{M_{\chi_1}^*\otimes M_{\chi_2}^*})\circ(\varphi_{\chi_1}\otimes\varphi_{\chi_2})
\end{equation*}
is an intertwining operator of type $\binom{\Phi(M_{\chi_1}^*\otimes M_{\chi_2}^*)}{V_{\chi_1}\,V_{\chi_2}}$. This shows that there are single-valued branches of the multivalued functions
\begin{equation*}
 \langle w_4', \cY_1(w_1,z_1)\cY_2(w_2,z_2)w_3\rangle
\end{equation*}
and
\begin{equation*}
 \langle w_4', \cY^1(\cY^2(w_1,z_1-z_2)w_2,z_2)w_3\rangle
\end{equation*}
that agree on any simply-connected set containing the intersection of the region $\vert z_1\vert>\vert z_2\vert>\vert z_1-z_2\vert>0$ with $\RR_+\times\RR_+$. Consequently, the two multivalued functions restrict to equal multivalued functions on the domain $\vert z_1\vert>\vert z_2\vert>\vert z_1-z_2\vert>0$.

The same kind of argument shows that it is possible to write any iterate
\begin{equation*}
 \langle w_4', \cY^1(\cY^2(w_1,z_0)w_2,z_2)w_3\rangle
\end{equation*}
of intertwining operators in $\cC_V$ as a product of intertwining operators. This establishes the existence of associativity isomorphisms in $\cC_V$.

Finally, establishing the coherence conditions for the vertex tensor category structure on $\cC_V$ requires the additional convergence condition found in \cite[Assumption 12.2]{HLZ8}:
\begin{itemize}
 \item Given modules $W_1$, $W_2$, $W_3$, $W_4$, $W_5$, $X_1$, $X_2$ in $\cC_V$ and intertwining operators $\cY_1$, $\cY_2$, $\cY_3$ of types $\binom{W_5}{W_1\,X_1}$, $\binom{X_1}{W_2\,X_2}$, $\binom{X_2}{W_3\,W_4}$, respectively, the double series
 \begin{align*}
  \langle w_5', \cY_1(w_1,z_1)\cY_2(w_2,z_2)\cY_3(w_3,z_3)w_4\rangle = \sum_{h_1,h_2\in\CC} \langle w_5', \cY_1(w_1,z_1)\pi_{h_1}\left(\cY_2(w_2,z_2)\pi_{h_2}\left(\cY_3(w_3,z_3)w_4\right)\right)\rangle
 \end{align*}
converges absolutely when $\vert z_1\vert>\vert z_2\vert>\vert z_3\vert>\vert z_4\vert>0$ for any $w_1\in W_1$, $w_2\in W_2$, $w_3\in W_3$, $w_4\in W_4$, and $w_5'\in W_5'$. Moreover, the convergent double series can be extended to a multivalued analytic function with only possible poles at $z_i=0,\infty$ ($i=1,2,3$) and $z_i=z_j$ ($i\neq j$). Near each singular point, the multivalued function can be expanded as a series having the same form as the expansion of a solution to a regular singular point differential equation at a singularity.
\end{itemize}
To verify this condition, we may again take $W_i=V_{\chi_i}$ for $i=1,2,3,4$, $W_5=\Phi(M_5)$, $X_j=\Phi(N_j)$ for $j=1,2$,
\begin{equation*}
 \cY_1=J^{M_5}_{M_{\chi_1}^*, N_1}(f_1)\circ(\varphi_{\chi_1}\otimes 1_{\Phi(N_1)})
\end{equation*}
for $f_1\in\hom_G(M_{\chi_1}^*\otimes N_1, M_5)$,
\begin{equation*}
 \cY_2=J^{N_1}_{M_{\chi_2}^*, N_2}(f_2)\circ(\varphi_{\chi_2}\otimes 1_{\Phi(N_2)})
\end{equation*}
for $f_2\in\hom_G(M_{\chi_2}^*\otimes N_2, N_1)$, and
\begin{equation*}
 \cY_3=J^{N_2}_{M_{\chi_3}^*, M_{\chi_4}^*}(f_3)\circ(\varphi_{\chi_3}\otimes\varphi_{\chi_4})
\end{equation*}
for $f_3\in\hom_G(M_{\chi_3}^*\otimes M_{\chi_4}^*, N_2)$. Then we have
\begin{align*}
  \langle w_5', \cY_1 & (w_1,z_1)\cY_2(w_2,z_2)\cY_3(w_3,z_3)w_4\rangle\nonumber\\
  & =\sum_{i,j,k,l} \langle w_5', f_1(m_{\chi_1,i}'\otimes f_2(m_{\chi_2,j}'\otimes f_3(m_{\chi_3,k}'\otimes m_{\chi_4,l}')))\otimes\nonumber\\
  &\hspace{5em} \otimes Y(m_{\chi_1, i}\otimes w_1, z_1)Y(m_{\chi_2,j}\otimes w_2, z_2)Y(m_{\chi_3, k}\otimes w_3, z_3)(m_{\chi_4,l}\otimes w_4)\rangle,
\end{align*}
using the usual notation for bases of $G$-modules and their duals. The desired convergence and expansion condition then follows directly from Proposition \ref{threevrtxops}.
\end{proof}

\subsection{Vertex tensor category structure implies a braided equivalence}\label{subsec:secondmaintheo}

In general it seems to be difficult to show that the linear maps $J^{M_3}_{M_1, M_2}$ are surjective (and thus calculate the fusion rules among the $V_\chi$) without using associativity of intertwining operators. However, in many cases, vertex tensor category structure on a suitable category of $V^G$-modules is known independently. Thus in this subsection we will work under the following assumption:
\begin{assum}
 There is a vertex tensor category $\cC$ (as in \cite{HLZ1}-\cite{HLZ8}) of grading-restricted generalized $V^G$-modules that contains all $V_\chi$ for $\chi\in\widehat{G}$.
\end{assum}

For any $z\in\CC^\times$, any morphism $f: M_1\otimes M_2\rightarrow M_3$ in $\rep_{A,F,\Omega}\,G$ induces a $P(z)$-intertwining map $I_f$ of type $\binom{\Phi(M_3)}{\Phi(M_1)\,\Phi(M_2)}$ by 
$$I_f=J^{M_3}_{M_1, M_2}(f)\left(\cdot, e^{\log z}\right)\cdot.$$
Then the universal property of $P(z)$-tensor products implies there is a unique $V^G$-module homomorphism
\begin{equation*}
 J_{P(z); M_1,M_2}: \Phi(M_1)\tens_{P(z)}\Phi(M_2)\rightarrow\Phi(M_1\otimes M_2)
\end{equation*}
such that
\begin{equation*}
 \overline{J_{P(z); M_1,M_2}}(w_1\boxtimes_{P(z)} w_2) = I_{1_{M_1\otimes M_2}}(w_1\otimes w_2)
\end{equation*}
for $w_1\in\Phi(M_1)$ and $w_2\in\Phi(M_2)$. We will set $\boxtimes=\tens_{P(1)}$ and $J_{M_1,M_2}=J_{P(1); M_1,M_2}$. An easy consequence of Proposition \ref{Jinjective} is the following:
\begin{propo}\label{Jsurjectiverema}
 For all modules $M_1$, $M_2$ in $\rep_{A,F,\Omega} G$ and all $z\in\CC^\times$, $J_{P(z); M_1, M_2}$ is surjective.
\end{propo}
\begin{proof}
 From the definitions,
 \begin{equation*}
  \overline{\Phi(f)\circ J_{P(z); M_1,M_2}} = I_f
 \end{equation*}
for any $G$-module homomorphism $f: M_1\otimes M_2\rightarrow M_3$. Thus $\Phi(f)\circ J_{P(z); M_1,M_2} = 0$ exactly when $I_f=0$, which occurs exactly when $f=0$ by Proposition \ref{Jsurjective}.

Now suppose $F: \Phi(M_1\otimes M_2)\rightarrow W$ is any morphism in $\cC$ such that $F\circ J_{P(z); M_1,M_2}=0$. Since $\Phi(M_1\otimes M_2)$ is semisimple, the image of $F$ is a semisimple $V^G$-module in $\cC_V$ and thus isomorphic to some $\Phi(M_3)$ because $\Phi$ is an equivalence of categories. Thus
\begin{equation*}
 F =i\circ\varphi\circ\Phi(f)
\end{equation*}
where $i$ is the inclusion of the image of $F$ into $W$, $\varphi$ is an isomorphism from $\Phi(M_3)$ to the image of $W$, and $f: M_1\otimes M_2\rightarrow M_3$ is some $G$-module homomorphism. Since $i\circ\varphi$ is injective, it follows that $\Phi(f)\circ J_{P(z); M_1,M_2}=0$, and hence $f=0$ and $F=0$, showing that $J_{P(z); M_1, M_2}$ is surjective.
\end{proof}

We claim that the morphisms $J_{P(z); M_1,M_2}$ determine a natural transformation $J_{P(z)}$ from $\boxtimes_{P(z)}\circ(\Phi\times\Phi)$ to $\Phi\circ\otimes$, that is, for all morphisms $f_1: M_1\rightarrow\widetilde{M}_1$ and $f_2: M_2\rightarrow\widetilde{M}_2$ in $\rep_{A,F,\Omega}\,G$, the diagram
\begin{equation*}
 \xymatrixcolsep{4pc}
 \xymatrix{
 \Phi(M_1)\tens_{P(z)}\Phi(M_2) \ar[d]^{\Phi(f_1)\tens_{P(z)}\Phi(f_2)} \ar[r]^(.55){J_{P(z); M_1,M_2}} & \Phi(M_1\otimes M_2) \ar[d]^{\Phi(f_1\otimes f_2)} \\
 \Phi(\widetilde{M}_1)\tens_{P(z)}\Phi(\widetilde{M}_2) \ar[r]^(.55){J_{P(z); \widetilde{M}_1,\widetilde{M}_2}} & \Phi(\widetilde{M}_1\otimes\widetilde{M}_2) \\
 }
\end{equation*}
commutes. To verify this, take $w_1=\sum_i m^{(1)}_i\otimes v^{(1)}_i\in \Phi(M_1)$ and $w_2=\sum_j m^{(2)}_j\otimes v^{(2)}_j\in\Phi(M_2)$. We have
\begin{align*}
 &\overline{J_{P(z); \widetilde{M}_1,\widetilde{M}_2}\circ(\Phi(f_1)\tens_{P(z)}\Phi(f_2))}(w_1\tens_{P(z)} w_2)\nonumber\\
 &\hspace{6em}= \overline{J_{P(z); \widetilde{M}_1, \widetilde{M}_2}}\left(\left(\sum_i f_1(m^{(1)}_i)\otimes v^{(1)}_i\right)\tens_{P(z)}\left(\sum_j f_2(m^{(2)}_j)\otimes v^{(2)}_j\right)\right)\nonumber\\
 &\hspace{6em} = \sum_{i,j} \left(f_1(m^{(1)}_i)\otimes f_2(m^{(2)}_j)\right)\otimes Y\left( v^{(1)}_i, e^{\log z}\right) v^{(2)}_j\nonumber\\
 & \hspace{6em} = \overline{(f_1\otimes f_2)\otimes 1_V}\left(\sum_{i,j} (m^{(1)}_i\otimes m^{(2)}_j)\otimes Y\left(v^{(1)}_i, e^{\log z}\right)v^{(2)}_j\right)\nonumber\\
 & \hspace{6em} = \overline{\Phi(f_1\otimes f_2)\circ J_{P(z); M_1,M_2}}\left(w_1\tens_{P(z)} w_2\right),
\end{align*}
as desired.

We can actually show that $\Phi$, the natural transformations $J_{P(z)}$, and the isomorphism $\varphi_1: V^G\rightarrow\Phi(\CC)$ determine an (\textit{a priori} lax) vertex tensor functor if we give $\rep_{A,F,\Omega} G$ the following vertex tensor category structure:
\begin{itemize}
 \item All $P(z)$-tensor product functors are the usual tensor product in $\rep_{A,F,\Omega} G$.
 
 \item For a continuous path $\gamma$ in $\CC^\times$ beginning at $z_1$ and ending at $z_2$, let $p\in\ZZ$ such that $\log z_1+2\pi i p$ is the branch of logarithm at $z_1$ determined by $\log z_2$ and the path $\gamma$. Then for $M_1$ in $\rep^{\alpha_1}\,G$ and $M_2$ in $\rep^{\alpha_2}\,G$, the parallel transport isomorphism
 \begin{equation*}
  T_{\gamma; M_1, M_2}: M_1\tens_{P(z_1)} M_2 = M_1\otimes M_2 \longrightarrow M_1\tens_{P(z_2)} M_2 = M_1\otimes M_2 
 \end{equation*}
is scalar multiplication by $\left(\Omega(\alpha_1,\alpha_2)\Omega(\alpha_2,\alpha_1)\right)^{-p}$.

\item For all $z\in\CC^\times$, the $P(z)$-unit and $P(z)$-braiding isomorphisms are the unit and braiding isomorphisms in $\rep_{A,F,\Omega}\,G$.

\item For $z_1,z_2\in\CC^\times$ such that $\vert z_1\vert>\vert z_2\vert>\vert z_1-z_2\vert>0$, the $P(z_1,z_2)$-associativity isomorphism is given by
\begin{equation*}
 \cA_{P(z_1,z_2); M_1,M_2,M_3}=\cA_{z_1,z_2}(\alpha_1,\alpha_2,\alpha_3)\cA_{\rep G; M_1,M_2,M_3}
\end{equation*}
(recall the notation of Proposition \ref{abintwalgassoc}) when each $M_i$ is in $\rep^{\alpha_i}\,G$ for $i=1,2,3$.
\end{itemize}
\begin{rema}\label{RepGvrtxtobraid}
 If we apply the recipe in \cite{HLZ8} (see also Section \ref{subsec:vrtxtenscat}) for obtaining a braided tensor category from the vertex tensor category structure on $\rep_{A,F,\Omega}\,G$, we get back the braided tensor category structure on $\rep_{A,F,\Omega}\,G$. (For the associativity isomorphisms, this uses Remark \ref{abintwalgassocreal}.)
\end{rema}

Now we have:
\begin{theo}\label{PhiVrtxTens}
The triple $(\Phi,\lbrace J_{P(z)}\rbrace_{z\in\CC^\times}, \varphi_1)$ is a lax vertex tensor functor from $\rep_{A,F,\Omega} G$ to $\cC$.
\end{theo}
\begin{proof}
In the proof, we will use $M_i$ to represent an object of $\rep^{\alpha_i}\,G$ for $i=1,2,3$, and we will represent typical vectors in $\Phi(M_1)$, $\Phi(M_2)$, and $\Phi(M_3)$ by $w_1=\sum_i m^{(1)}_i\otimes v^{(1)}_i$, $w_2=\sum_j m^{(2)}_j\otimes v^{(2)}_j$, and $w_3=\sum_k m^{(3)}_k\otimes v^{(3)}_k$, respectively.

 First we show that $J_{P(z)}$ is compatible with parallel transport isomorphisms in the sense that the diagram \eqref{partranscompat} commutes for all continuous paths $\gamma$ in $\CC^\times$ beginning at $z_1$ and ending at $z_2$. To verify this, suppose that $\log z_1+2\pi i p$ is the branch of logarithm at $z_1$ determined by $\log z_2$ and the path $\gamma$. Then
\begin{align*}
 & \overline{J_{P(z_2); M_1,M_2}\circ T_{\gamma; \Phi(M_1),\Phi(M_2)}}(w_1\tens_{P(z_1)} w_2) = \overline{J_{P(z_2); M_1,M_2}}\left(\cY_{\tens_{P(z_2)}}(w_1, e^{\log z_1+2\pi i p})w_2\right)\nonumber\\
 &\hspace{2em} = e^{(\log z_1+2\pi i p-\log z_2)L(0)}\overline{J_{P(z_2); M_1,M_2}}\left(( e^{-(\log z_1+2\pi i p-\log z_2)L(0)} w_1)\tens_{P(z_2)} ( e^{-(\log z_1+2\pi i p-\log z_2)L(0)} w_2)\right)\nonumber\\
 &\hspace{2em} =\sum_{i,j} (m^{(1)}_i\otimes m^{(2)}_j)\otimes Y\left(v^{(1)}_i, e^{\log z_1+2\pi i p}\right)v^{(2)}_j\nonumber\\
 & \hspace{2em} = \sum_{i,j} \left(\Omega(\alpha_1,\alpha_2)\Omega(\alpha_2,\alpha_1)\right)^{-p} (m^{(1)}_i\otimes m^{(2)}_j)\otimes Y\left(v^{(1)}_i, e^{\log z_1}\right)v^{(2)}_j\nonumber\\
 &\hspace{2em} =\overline{\Phi(T_{\gamma; M_1,M_2})\circ J_{P(z_1); M_1,M_2}}(w_1\tens_{P(z_1)} w_2),
\end{align*}
using Corollary \ref{abintalgmono} and the definition of $T_{\gamma; M_1,M_2}$.

Next we need to show that $J_{P(z)}$ is compatible with $\varphi_1$ and the unit isomorphisms in the sense that the diagrams \eqref{unitcompat} commute for any $M$ in $\rep_{A,F,\Omega} G$. To prove this, take $w=\sum_i m_i\otimes v_i\in\Phi(M)$ and $u\in V^G$. Then we have
\begin{align*}
 \overline{\Phi(l_M)\circ J_{P(z);\CC,M}\circ(\varphi_1\tens_{P(z)} 1_{\Phi(M)})}(u\tens_{P(z)} w) & = \overline{\Phi(l_M)\circ J_{P(z);\CC,M}}\left((1\otimes(1\otimes u))\tens_{P(z)} \sum_i m_i\otimes v_i\right)\nonumber\\
 &\hspace{-5em} =\overline{\Phi(l_M)}\left(\sum_{i} (1\otimes m_i)\otimes Y(1\otimes u, e^{\log z})v_i\right)=\sum_i m_i\otimes Y(u,z)v_i\nonumber\\
 &\hspace{-5em}= Y_{\Phi(M)} (u, z)w =\overline{l_{P(z); \Phi(m)}}(u\tens_{P(z)} w),
\end{align*}
where we can write $e^{\log z}=z$ inside $Y$ because $Y(u,x)$ contains only integral powers of $x$ when $u\in V^G\subseteq V^0$.
We also have
\begin{align*}
 \overline{\Phi(r_M)\circ J_{P(z); M, \CC}\circ(1_{\Phi(M)}\tens_{P(z)} \varphi_1)} & (w\tens_{P(z)} u)  = \overline{\Phi(r_M)\circ J_{P(z); M, \CC}}\left(\left(\sum_i m_i\otimes v_i\right)\tens_{P(z)}(1\otimes(1\otimes u))\right)\nonumber\\
 &\hspace{-6em} =\overline{\Phi(r_M)}\left(\sum_i (m_i\otimes 1)\otimes Y(v_i, e^{\log z})(1\otimes u)\right) = \sum_i m_i\otimes e^{z L(-1)} Y(u, e^{\log z+\pi i})v_i\nonumber\\
 & \hspace{-6em} =e^{z L(-1)} Y_{\Phi(M)}(u, -z)w = \overline{r_{P(z); \Phi(M)}}(w\tens_{P(z)} u),
\end{align*}
where we have used Proposition \ref{abintalgskewsym} in the case $\alpha_2=0$.

Now we show that $J_{P(z)}$ is compatible with the $P(z)$-braiding isomorphisms in $\rep_{A,F,\Omega} G$ and $\cC$, that is, the diagram \eqref{braidcompat} commutes. To prove this, we calculate
\begin{align*}
 \overline{J_{P(-z); M_2,M_1}\circ\cR_{P(z);\Phi(M_1),\Phi(M_2)}}& (w_1\tens_{P(z)} w_2) = \overline{J_{P(-z); M_2,M_1}}\left(e^{z L(-1)}\cY_{\tens_{P(-z)}}(w_2, e^{\log z+\pi i})w_1\right)\nonumber\\
 & \hspace{-12em} = e^{zL(-1)} e^{(\log z+\pi i-\log(-z))L(0)}\cdot\nonumber\\
 &\hspace{-8em}\cdot\overline{J_{P(-z);M_2, M_1}}\left(( e^{-(\log z+\pi i-\log(-z))L(0)}w_2)\tens_{P(-z)}( e^{-(\log z+\pi i-\log(-z))L(0)}w_1)\right)\nonumber\\
 & \hspace{-12em} =\sum_{i,j} (m^{(2)}_j\otimes m^{(1)}_i)\otimes e^{z L(-1)} Y(v^{(2)}_j, e^{\log z+\pi i})v^{(1)}_i=\sum_{i,j} \Omega(\alpha_1,\alpha_2)^{-1}(m^{(2)}_j\otimes m^{(1)}_i)\otimes Y(v^{(1)}_i, e^{\log z})v^{(2)}_j\nonumber\\
 & \hspace{-12em} =\overline{\Phi(\cR_{M_1,M_2})}\left(\sum_{i,j} (m^{(1)}_i\otimes m^{(2)}_j)\otimes Y(v^{(1)}_i, e^{\log z})v^{(2)}_j\right) =  \overline{\Phi(\cR_{M_1,M_2})\circ J_{P(z); M_1,M_2}}(w_1\tens_{P(z)} w_2),
\end{align*}
where we have used skew-symmetry of $Y$, Proposition \ref{abintalgskewsym}.

Finally, we have to show that $J_{P(z)}$ is compatible with the $P(z_1,z_2)$-associativity isomorphisms in $\rep_{A,F,\Omega}G$ and $\cC$, that is, the diagram \eqref{assoccompat} commutes. Here, the calculation uses Proposition \ref{abintwalgassoc}:
 \begin{align*}
  &\overline{J_{P(z_2); M_1\otimes M_2, M_3}\circ(J_{P(z_1-z_2);M_1,M_2}\tens_{P(z_2)} 1_{\Phi(M_3)})\circ\cA_{P(z_1,z_2);\Phi(M_1),\Phi(M_2),\Phi(M_3)}}(w_1\tens_{P(z_1)}(w_2\tens_{P(z_2)} w_3))\nonumber\\
  &\hspace{2.5em} =\overline{J_{P(z_2); M_1\otimes M_2, M_3}\circ(J_{P(z_1-z_2);M_1,M_2}\tens_{P(z_2)} 1_{\Phi(M_3)})}((w_1\tens_{P(z_1-z_2)} w_2)\tens_{P(z_2)} w_3)\nonumber\\
  &\hspace{2.5em} =\overline{J_{P(z_2); M_1\otimes M_2, M_3}}\left(\left(\sum_{i,j} (m^{(1)}_i\otimes m^{(2)}_j)\otimes Y(v^{(1)}_i, e^{\log(z_1-z_2)})v^{(2)}_j\right)\tens_{P(z_2)} w_3\right)\nonumber\\
  &\hspace{2.5em} =\sum_{i,j,k} ((m^{(1)}_i\otimes m^{(2)}_j)\otimes m^{(3)}_k)\otimes Y\left(Y\left(v^{(1)}_i, e^{\log(z_1-z_2)}\right)v^{(2)}_j, e^{\log z_2}\right)v^{(3)}_k\nonumber\\
  &\hspace{2.5em} =\sum_{i,j,k} \cA_{z_1,z_2}(\alpha_1,\alpha_2,\alpha_3)\left((m^{(1)}_i\otimes m^{(2)}_j)\otimes m^{(3)}_k\right)\otimes Y(v^{(1)}_i, e^{\log z_1})Y(v^{(2)}_j, e^{\log z_2})v^{(3)}_k\nonumber\\
  &\hspace{2.5em} =\overline{\Phi(\cA_{P(z_1,z_2); M_1,M_2,M_3})}\left(\sum_{i,j,k} \left(m^{(1)}_i\otimes(m^{(2)}_j\otimes m^{(3)}_k)\right)\otimes Y(v^{(1)}_i, e^{\log z_1})Y(v^{(2)}_j, e^{\log z_2})v^{(3)}_k\right)\nonumber\\
  &\hspace{2.5em} =\overline{\Phi(\cA_{P(z_1,z_2); M_1,M_2,M_3})\circ J_{P(z_1); M_1,M_2\otimes M_3}}\left(w_1\tens_{P(z_1)}\sum_{j,k} (m^{(2)}_j\otimes m^{(3)}_k)\otimes Y(v^{(2)}_j, e^{\log z_2})v^{(3)}_k\right)\nonumber\\
  &\hspace{2.5em} =\overline{\Phi(\cA_{P(z_1,z_2); M_1,M_2,M_3})\circ J_{P(z_1); M_1,M_2\otimes M_3}\circ(1_{\Phi(M_1)}\tens_{P(z_1)} J_{P(z_2); M_2,M_3})}(w_1\tens_{P(z_1)}(w_2\tens_{P(z_2)} w_3)),
 \end{align*}
as desired. This completes the proof that $(\Phi, \lbrace J_{P(z)}\rbrace_{z\in\CC^\times}, \varphi_1)$ is a (possibly lax) vertex tensor functor in the sense of \cite{CKM}.
\end{proof}

Now the previous theorem, Theorem \ref{vrtxtobraidfunctor}, and Remark \ref{RepGvrtxtobraid} immediately yield:
\begin{corol}\label{PhiLaxTens}
 The triple $(\Phi,J,\varphi_1)$ is a lax braided tensor functor from $\rep_{A,F,\Omega} G$ to $\cC$.
\end{corol}

Now we can complete the proof that $\Phi$ gives a braided tensor equivalence by showing that $J_{M_1,M_2}$ is injective for any $G$-modules $M_1$, $M_2$ in $\rep_{A,F,\Omega} G$. The proof is categorical and heavily uses compatibility of $J$ with associativity and unit isomorphisms:
\begin{theo}\label{Jinjective}
 For any $G$-modules $M_1$, $M_2$ in $\rep_{A,F,\Omega} G$, the homomorphism $J_{M_1,M_2}: \Phi(M_1)\tens\Phi(M_2)\rightarrow\Phi(M_1\otimes M_2)$ is injective.
\end{theo}
\begin{proof}
 Let $K$ be the kernel of $J_{M_1, M_2}$ and let $k: K\rightarrow\Phi(M_1)\tens\Phi(M_2)$ be the canonical embedding. Consider the $V^G$-module homomorphism $F$ defined by the composition
 \begin{align*}
  \Phi & (M_1\otimes M_1^*)  \tens K\xrightarrow{1_{\Phi(M_1\otimes M_1^*)}\tens k} \Phi(M_1\otimes M_1^*)\tens(\Phi(M_1)\tens\Phi(M_2))\nonumber\\
  & \xrightarrow{\cA_{\Phi(M_1\otimes M_1^*),\Phi(M_1),\Phi(M_2)}} (\Phi(M_1\otimes M_1^*)\tens\Phi(M_1))\tens\Phi(M_2)\xrightarrow{J_{M_1\otimes M_1^*,M_1}\tens 1_{\Phi(M_2)}} \Phi((M_1\otimes M_1^*)\otimes M_1)\boxtimes\Phi(M_2)\nonumber\\
 & \xrightarrow{\Phi(\cA_{M_1,M_1^*,M_1}^{-1})\tens 1_{\Phi(M_2)}} \Phi(M_1\otimes(M_1^*\otimes M_1))\tens\Phi(M_2)\xrightarrow{\Phi(1_{M_1}\otimes e_{M_1})\tens 1_{\Phi(M_2)}} \Phi(M_1\otimes\CC)\tens\Phi(M_2)\nonumber\\
 &\xrightarrow{\Phi(r_{M_1})\tens 1_{\Phi(M_2)}} \Phi(M_1)\tens\Phi(M_2).
 \end{align*}
The theorem will follow from the following two claims:

\bigskip

\textbf{Claim 1:} $F\circ(\Phi(i_{M_1})\tens 1_K)\circ(\varphi_1\tens 1_K)\circ l_{K}^{-1} =k$.

\bigskip

\textbf{Claim 2:} $F\circ(J_{M_1,M_1^*}\tens 1_K)=0$.

\bigskip

\noindent For, since we have already proved $J_{M_1, M_1^*}$ is surjective, and since the functor $\cdot\tens K$ is right exact, $J_{M_1,M_1^*}\tens 1_K$ is surjective and Claim 2 implies $F=0$. Then Claim 1 will imply $k=0$, and hence $K=0$.

Now to verify Claim 1, by the naturality of the left unit isomorphisms in $\cC$, the indicated composition is
\begin{align*}
& K\xrightarrow{k}  \Phi(M_1)\tens\Phi(M_2)\xrightarrow{l_{\Phi(M_1)\tens\Phi(M_2)}^{-1}} V^G\tens(\Phi(M_1)\tens\Phi(M_2))\xrightarrow{\varphi_1\tens 1_{\Phi(M_1)\tens\Phi(M_2)}} \Phi(\CC)\tens(\Phi(M_1)\tens\Phi(M_2))\nonumber\\
 & \xrightarrow{\Phi(i_{M_1})\tens 1_{\Phi(M_1)\tens\Phi(M_2)}} \Phi(M_1\tens M_1^*)\tens(\Phi(M_1)\tens\Phi(M_2)) \xrightarrow{\cA_{\Phi(M_1\otimes M_1^*),\Phi(M_1),\Phi(M_2)}} (\Phi(M_1\otimes M_1^*)\tens\Phi(M_1))\tens\Phi(M_2)\nonumber\\
 &\xrightarrow{J_{M_1\otimes M_1^*,M_1}\tens 1_{\Phi(M_2)}} \Phi((M_1\otimes M_1^*)\otimes M_1)\tens\Phi(M_2) \xrightarrow{\Phi(\cA_{M_1,M_1^*,M_1}^{-1})\tens 1_{\Phi(M_2)}} \Phi(M_1\otimes(M_1^*\otimes M_1))\tens\Phi(M_2)\nonumber\\
 &\xrightarrow{\Phi(1_{M_1}\otimes e_{M_1})\tens 1_{\Phi(M_2)}} \Phi(M_1\otimes\CC)\tens\Phi(M_2) \xrightarrow{\Phi(r_{M_1})\tens 1_{\Phi(M_2)}} \Phi(M_1)\tens\Phi(M_2).
\end{align*}
Next we can use the naturality of the associativity isomorphisms and the triangle in $\cC$ (see for instance \cite[Lemma XI.2.2]{Ka}) to rewrite this composition as
\begin{align*}
 & K\xrightarrow{k} \Phi(M_1)\tens\Phi(M_2)\xrightarrow{l_{\Phi(M_1)}^{-1}\tens 1_{\Phi(M_2)}} (V^G\tens\Phi(M_1))\tens\Phi(M_2)\xrightarrow{(\varphi_1\tens 1_{\Phi(M_1)})\tens 1_{\Phi(M_2)}} (\Phi(\CC)\tens\Phi(M_1))\tens\Phi(M_2)\nonumber\\
 & \xrightarrow{(\Phi(i_{M_1})\tens 1_{\Phi(M_1)})\tens 1_{\Phi(M_2)}} (\Phi(M_1\otimes M_1^*)\tens\Phi(M_1))\tens\Phi(M_2) \xrightarrow{J_{M_1\otimes M_1^*,M_1}\tens 1_{\Phi(M_2)}} \Phi((M_1\otimes M_1^*)\otimes M_1)\otimes\Phi(M_2)\nonumber\\
 & \xrightarrow{\Phi(\cA_{M_1,M_1^*,M_1}^{-1})\tens 1_{\Phi(M_2)}} \Phi(M_1\otimes(M_1^*\otimes M_1))\tens\Phi(M_2)\xrightarrow{\Phi(1_{M_1}\otimes e_{M_1})\tens 1_{\Phi(M_2)}} \Phi(M_1\otimes\CC)\tens\Phi(M_2)\nonumber\\
 &\xrightarrow{\Phi(r_{M_1})\tens 1_{\Phi(M_2)}} \Phi(M_1)\tens\Phi(M_2).
\end{align*}
Now the naturality of $J$ implies that
\begin{equation*}
 J_{M_1\otimes M_1^*, M_1}\circ(\Phi(i_{M_1})\tens 1_{\Phi(M_1)})=\Phi(i_{M_1}\otimes 1_{M_1})\circ J_{\CC, M_1}.
\end{equation*}
Moreover, the compatibility of $\varphi_1$ and $J$ with left unit isomorphisms shows that we get the composition
\begin{align*}
 &K\xrightarrow{k}\Phi(M_1)\tens\Phi(M_2)\xrightarrow{\Phi(l_{M_1}^{-1})\tens 1_{\Phi(M_2)}} \Phi(\CC\tens M_1)\tens\Phi(M_2)\xrightarrow{\Phi(i_{M_1}\tens 1_{M_1})\tens 1_{\Phi(M_2)}} \Phi((M_1\otimes M_1^*)\otimes M_1)\tens\Phi(M_2)\nonumber\\
 & \xrightarrow{\Phi(\cA_{M_1,M_1^*,M_1}^{-1})\tens 1_{\Phi(M_2)}} \Phi(M_1\otimes(M_1^*\otimes M_1))\tens\Phi(M_2)\xrightarrow{\Phi(1_{M_1}\otimes e_{M_1})\tens 1_{\Phi(M_2)}} \Phi(M_1\otimes\CC)\tens\Phi(M_2)\nonumber\\
 &\xrightarrow{\Phi(r_{M_1})\tens 1_{\Phi(M_2)}} \Phi(M_1)\tens\Phi(M_2).
\end{align*}
Now the rigidity of $M_1$ in $\rep_{A,F,\Omega}\,G$ implies that the entire composition collapses to $k$, completing the proof of Claim 1.

Now for Claim 2, we are considering the composition
\begin{align*}
 & (\Phi(M_1)\tens\Phi(M_1^*))\tens K\xrightarrow{J_{M_1,M_1^*}\tens k}\Phi(M_1\tens M_1^*)\tens(\Phi(M_1)\tens\Phi(M_2))\nonumber\\
 & \xrightarrow{\cA_{\Phi(M_1\otimes M_1^*),\Phi(M_1),\Phi(M_2)}} (\Phi(M_1\otimes M_1^*)\tens\Phi(M_1))\tens\Phi(M_2)\xrightarrow{J_{M_1\otimes M_1^*,M_1}\tens 1_{\Phi(M_2)}} \Phi((M_1\otimes M_1^*)\otimes M_1)\tens\Phi(M_2)\nonumber\\
 & \xrightarrow{\Phi(\cA_{M_1,M_1^*,M_1}^{-1})\tens 1_{\Phi(M_2)}} \Phi(M_1\otimes(M_1^*\otimes M_1))\tens\Phi(M_2)\xrightarrow{\Phi(1_{M_1}\otimes e_{M_1})\tens 1_{\Phi(M_2)}} \Phi(M_1\otimes\CC)\tens\Phi(M_2)\nonumber\\
 &\xrightarrow{\Phi(r_{M_1})\tens 1_{\Phi(M_2)}} \Phi(M_1)\tens\Phi(M_2).
\end{align*}
The idea is to use compatibility of $J$ with the associativity and unit isomorphisms to move the $J_{M_1,M_1^*}$ at the beginning on the left over to the right. The naturality of the associativity isomorphism in $\cC$ implies that our composition is
\begin{align*}
 & (\Phi(M_1)\tens\Phi(M_1^*))\tens K\xrightarrow{1_{\Phi(M_1)\tens\Phi(M_1^*)}\tens k} (\Phi(M_1)\tens\Phi(M_1^*))\tens(\Phi(M_1)\tens\Phi(M_2))\nonumber\\
 & \xrightarrow{\cA_{\Phi(M_1)\tens\Phi(M_1^*),\Phi(M_1),\Phi(M_2)}} ((\Phi(M_1)\tens\Phi(M_1^*))\tens\Phi(M_1))\tens\Phi(M_2)\nonumber\\
 & \xrightarrow{(J_{M_1,M_1^*}\tens 1_{\Phi(M_1)})\tens 1_{\Phi(M_2)}} (\Phi(M_1\otimes M_1^*)\tens\Phi(M_1))\tens\Phi(M_2) \xrightarrow{J_{M_1\otimes M_1^*, M_1}\tens 1_{\Phi(M_2)}} \Phi((M_1\otimes M_1^*)\tens M_1)\tens\Phi(M_2)\nonumber\\
 & \xrightarrow{\Phi(\cA_{M_1,M_1^*,M_1}^{-1})\tens 1_{\Phi(M_2)}} \Phi(M_1\otimes(M_1^*\otimes M_1))\tens\Phi(M_2)\xrightarrow{\Phi(1_{M_1}\otimes e_{M_1})\tens 1_{\Phi(M_2)}} \Phi(M_1\otimes\CC)\tens\Phi(M_2)\nonumber\\
 &\xrightarrow{\Phi(r_{M_1})\tens 1_{\Phi(M_2)}} \Phi(M_1)\tens\Phi(M_2).
\end{align*}
Now by the compatibility of $J$ with the associativity isomorphisms, we can rewrite the third, fourth, and fifth arrows above:
\begin{align*}
 & (\Phi(M_1)\tens\Phi(M_1^*))\tens K\xrightarrow{1_{\Phi(M_1)\tens\Phi(M_1^*)}\tens k} (\Phi(M_1)\tens\Phi(M_1^*))\tens(\Phi(M_1)\tens\Phi(M_2))\nonumber\\
 & \xrightarrow{\cA_{\Phi(M_1)\tens\Phi(M_1^*),\Phi(M_1),\Phi(M_2)}} ((\Phi(M_1)\tens\Phi(M_1^*))\tens\Phi(M_1))\tens\Phi(M_2)\nonumber\\
 & \xrightarrow{\cA_{\Phi(M_1),\Phi(M_1^*),\Phi(M_1)}^{-1}\tens 1_{\Phi(M_2)}} (\Phi(M_1)\tens(\Phi(M_1^*)\tens\Phi(M_1)))\tens\Phi(M_2)\nonumber\\ & \xrightarrow{(1_{\Phi(M_1)}\tens J_{M_1^*, M_1})\tens 1_{\Phi(M_2)}} (\Phi(M_1)\tens\Phi(M_1^*\otimes M_1))\tens\Phi(M_2) \xrightarrow{J_{M_1,M_1^*\otimes M_1}\tens 1_{\Phi(M_2)}} \Phi(M_1\otimes(M_1^*\otimes M_1))\tens\Phi(M_2)\nonumber\\
 & \xrightarrow{\Phi(1_{M_1}\otimes e_{M_1})\tens 1_{\Phi(M_2)}} \Phi(M_1\otimes\CC)\tens\Phi(M_2) \xrightarrow{\Phi(r_{M_1})\tens 1_{\Phi(M_2)}} \Phi(M_1)\tens\Phi(M_2).
\end{align*}
Next, the naturality of $J$ implies that
\begin{equation*}
 \Phi(1_{M_1}\tens e_{M_1})\circ J_{M_1,M_1^*\otimes M_1}= J_{M_1,\CC}\circ(1_{\Phi(M_1)}\tens\Phi(e_{M_1}));
\end{equation*}
using this plus the compatibility of $J$ and $\varphi_1$ with the right unit isomorphisms, we get
\begin{align*}
 & (\Phi(M_1)\tens\Phi(M_1^*))\tens K\xrightarrow{1_{\Phi(M_1)\tens\Phi(M_1^*)}\tens k} (\Phi(M_1)\tens\Phi(M_1^*))\tens(\Phi(M_1)\tens\Phi(M_2))\nonumber\\
 & \xrightarrow{\cA_{\Phi(M_1)\tens\Phi(M_1^*),\Phi(M_1),\Phi(M_2)}} ((\Phi(M_1)\tens\Phi(M_1^*))\tens\Phi(M_1))\tens\Phi(M_2)\nonumber\\
 & \xrightarrow{\cA_{\Phi(M_1),\Phi(M_1^*),\Phi(M_1)}^{-1}\tens 1_{\Phi(M_2)}} (\Phi(M_1)\tens(\Phi(M_1^*)\tens\Phi(M_1)))\tens\Phi(M_2)\nonumber\\ & \xrightarrow{(1_{\Phi(M_1)}\tens J_{M_1^*, M_1})\tens 1_{\Phi(M_2)}} (\Phi(M_1)\tens\Phi(M_1^*\otimes M_1))\tens\Phi(M_2)\xrightarrow{(1_{\Phi(M_1)}\tens\Phi(e_{M_1}))\tens 1_{\Phi(M_2)}} (\Phi(M_1)\tens\Phi(\CC))\tens\Phi(M_2)\nonumber\\
 & \xrightarrow{(1_{\Phi(M_1)}\tens\varphi_1^{-1})\tens 1_{\Phi(M_2)}} (\Phi(M_1)\tens V^G)\tens\Phi(M_2)\xrightarrow{r_{\Phi(M_1)}\tens 1_{\Phi(M_2)}} \Phi(M_1)\tens\Phi(M_2).
\end{align*}
Now by the triangle axiom in $\cC$,
\begin{equation*}
 r_{\Phi(M_1)}\tens 1_{\Phi(M_2)}= (1_{\Phi(M_1)}\tens l_{\Phi(M_2)})\circ\cA^{-1}_{\Phi(M_1),V^G,\Phi(M_2)};
\end{equation*}
then by naturality we can move this new associativity isomorphism back over three arrows in our composition and rewrite the resulting composition of three associativity isomorphisms as a composition of two using the pentagon identity:
\begin{align*}
 & (\Phi(M_1)\tens\Phi(M_1^*))\tens K\xrightarrow{1_{\Phi(M_1)\tens\Phi(M_1^*)}\tens k} (\Phi(M_1)\tens\Phi(M_1^*))\tens(\Phi(M_1)\tens\Phi(M_2))\nonumber\\
 & \xrightarrow{\cA^{-1}_{\Phi(M_1),\Phi(M_1^*),\Phi(M_1)\tens\Phi(M_2)}} \Phi(M_1)\tens(\Phi(M_1^*)\tens(\Phi(M_1)\tens\Phi(M_2)))\nonumber\\
 & \xrightarrow{1_{\Phi(M_1)}\tens\cA_{\Phi(M_1^*),\Phi(M_1),\Phi(M_2)}} \Phi(M_1)\tens((\Phi(M_1^*)\tens\Phi(M_1))\tens\Phi(M_2))\nonumber\\
 & \xrightarrow{1_{\Phi(M_1)}\tens (J_{M_1^*, M_1}\tens 1_{\Phi(M_2)})} \Phi(M_1)\tens(\Phi(M_1^*\otimes M_1)\tens\Phi(M_2))\xrightarrow{1_{\Phi(M_1)}\tens(\Phi(e_{M_1})\tens 1_{\Phi(M_2)})} \Phi(M_1)\tens(\Phi(\CC)\tens\Phi(M_2))\nonumber\\
 & \xrightarrow{1_{\Phi(M_1)}\tens(\varphi_1^{-1}\tens 1_{\Phi(M_2)})} \Phi(M_1)\tens (V^G\tens\Phi(M_2))\xrightarrow{1_{\Phi(M_1)}\tens l_{\Phi(M_2)}} \Phi(M_1)\tens\Phi(M_2).
\end{align*}
We again use the compatibility of $J$ and $\varphi_1$ with now the left unit isomorphisms and the naturality of $J$ to rewrite the last three arrows above:
\begin{align*}
 & (\Phi(M_1)\tens\Phi(M_1^*))\tens K\xrightarrow{1_{\Phi(M_1)\tens\Phi(M_1^*)}\tens k} (\Phi(M_1)\tens\Phi(M_1^*))\tens(\Phi(M_1)\tens\Phi(M_2))\nonumber\\
 & \xrightarrow{\cA^{-1}_{\Phi(M_1),\Phi(M_1^*),\Phi(M_1)\tens\Phi(M_2)}} \Phi(M_1)\tens(\Phi(M_1^*)\tens(\Phi(M_1)\tens\Phi(M_2)))\nonumber\\
 & \xrightarrow{1_{\Phi(M_1)}\tens\cA_{\Phi(M_1^*),\Phi(M_1),\Phi(M_2)}} \Phi(M_1)\tens((\Phi(M_1^*)\tens\Phi(M_1))\tens\Phi(M_2))\nonumber\\
 & \xrightarrow{1_{\Phi(M_1)}\tens (J_{M_1^*, M_1}\tens 1_{\Phi(M_2)})} \Phi(M_1)\tens(\Phi(M_1^*\otimes M_1)\tens\Phi(M_2)) \xrightarrow{1_{\Phi(M_1)}\tens J_{M_1^*\otimes M_1, M_2}} \Phi(M_1)\tens\Phi((M_1^*\otimes M_1)\otimes M_2)\nonumber\\
& \xrightarrow{1_{\Phi(M_1)}\tens\Phi(e_{M_1}\tens 1_{M_2})} \Phi(M_1)\tens\Phi(\CC\tens M_2)\xrightarrow{1_{\Phi(M_1)}\tens\Phi(l_{M_2})} \Phi(M_1)\tens\Phi(M_2).
\end{align*}
Finally, we apply the naturality of the associativity isomorphisms to the first two arrows above, and we apply the compatibility of $J$ with the associativity isomorphisms to the third, fourth, and fifth arrows, to get:
\begin{align*}
& (\Phi(M_1)\tens\Phi(M_1^*))\tens K\xrightarrow{\cA^{-1}_{\Phi(M_1),\Phi(M_1^*), K}} \Phi(M_1)\tens(\Phi(M_1^*)\tens K)\nonumber\\ & \xrightarrow{1_{\Phi(M_1)}\tens(1_{\Phi(M_1^*)}\tens k)} \Phi(M_1)\tens(\Phi(M_1^*)\tens(\Phi(M_1)\tens\Phi(M_2)))\nonumber\\ & \xrightarrow{1_{\Phi(M_1)}\tens(1_{\Phi(M_1^*)}\tens J_{M_1,M_2})} \Phi(M_1)\tens(\Phi(M_1^*)\tens\Phi(M_1\otimes M_2)) \xrightarrow{1_{\Phi(M_1)}\tens J_{M_1^*, M_1\otimes M_2}} \Phi(M_1)\tens\Phi(M_1^*\otimes(M_1\otimes M_2))\nonumber\\
& \xrightarrow{1_{\Phi(M_1)}\tens\Phi(\cA_{M_1^*,M_1,M_2})} \Phi(M_1)\tens\Phi((M_1^*\otimes M_1)\otimes M_2) \xrightarrow{1_{\Phi(M_1)}\tens\Phi(e_{M_1}\tens 1_{M_2})} \Phi(M_1)\tens\Phi(\CC\tens M_2)\nonumber\\
& \xrightarrow{1_{\Phi(M_1)}\tens\Phi(l_{M_2})} \Phi(M_1)\tens\Phi(M_2).
\end{align*}
But this composition is $0$ because by definition $J_{M_1,M_2}\circ k=0$. This completes the proof of Claim 2 and hence of the theorem.
\end{proof}

Now as an immediate consequence of Proposition \ref{Jsurjectiverema}, Corollary \ref{PhiLaxTens} and Theorem \ref{Jinjective}, and Proposition \ref{PhiAbEquiv} we get:
\begin{corol}\label{maintheorem}
 The triple $(\Phi, J, \varphi_1)$ is a braided tensor functor from $\rep_{A,F,\Omega}\,G$ to $\cC$. Moreover, $(\Phi,J,\varphi_1)$ induces a braided tensor equivalence between $\rep_{A,F,\Omega}\,G$ and $\cC_V$.
\end{corol}

\begin{rema}
 The natural transformation $J$, shown to be an isomorphism in Proposition \ref{Jsurjective}, Proposition \ref{Jsurjectiverema}, and Theorem \ref{Jinjective}, is needed to show that $\cC_V$ is actually closed under tensor products and hence is a braided tensor  category. Note that since the proof was independent of $\cC$, tensor products in $\cC_V$ do not depend on the category of $V^G$-modules under consideration.
\end{rema}

\begin{rema}
 The fact that $J$ is an natural isomorphism, together with the compatibility of each $J_{P(z)}$ with parallel transport isomorphisms shown in the proof of Theorem \ref{PhiVrtxTens}, shows that $J_{P(z)}$ is a natural isomorphism for each $z\in\CC^\times$. Specifically, each $J_{P(z); M_1,M_2}$ is an isomorphism because
 \begin{equation*}
  J_{P(z); M_1, M_2}= \Phi(T_{z\to1; M_1,M_2})^{-1}\circ J_{M_1,M_2}\circ T_{z\to 1; \Phi(M_1), \Phi(M_2)}
 \end{equation*}
is a composition of isomorphisms. Thus we can view $(\Phi,\lbrace J_{P(z)}\rbrace_{z\in\CC^\times},\varphi_1)$ as an equivalence of vertex tensor categories, at least as far as the $P(z)$-tensor structures are concerned.
\end{rema}

\subsection{Examples and applications}\label{sec:exams}

Here we illustrate Theorem \ref{fusionruletovrtxtens} and Corollary \ref{maintheorem} with examples to show how they can be used. Theorem \ref{fusionruletovrtxtens} can be used in situations where fusion rules among $V^G$-modules are known \textit{a priori}, and Corollary \ref{maintheorem} can be used in situations where it is known \textit{a priori} that $V^G$ has a suitable vertex tensor category of modules.

\begin{exam}
In this example and the next, we discuss compact automorphism groups of the lattice vertex operator algebra $V_Q$ where $Q=\ZZ\alpha$ is the $\mathfrak{sl}_2$-root lattice, that is, $\langle\alpha,\alpha\rangle =2$. More generally, we consider the abelian intertwining algebra $V_P$, where $P=\ZZ\frac{\alpha}{2}$ is the $\mathfrak{sl}_2$-weight lattice, into which $V_Q$ embeds by \cite[Theorem 12.24]{DL}. 

First we consider the action of the compact abelian group $U(1)$ on $V_Q$ and $V_P$ and show how in these cases Theorem \ref{fusionruletovrtxtens} and Corollary \ref{maintheorem} recover familiar properties of the rank-one Heisenberg vertex operator algebra $\mathcal{H}$ and its modules. The vertex operator algebra $\mathcal{H}$ appears as a subalgebra of $V_Q$ and $V_P$, and these algebras decompose as $\mathcal{H}$-modules as follows:
\begin{equation*}
 V_Q=\bigoplus_{\lambda\in Q} \mathcal{F}_\lambda,\hspace{2em} V_P=\bigoplus_{\lambda\in P} \mathcal{F}_\lambda,
\end{equation*}
where $\mathcal{F}_\lambda$ is the irreducible Heisenberg Fock module of highest weight $\lambda$. Then there is an injective homomorphism $\rho: U(1)\cong\RR\alpha/P\rightarrow\mathrm{Aut}\,V_Q$ such that
\begin{equation*}
 \rho(\beta+P)\vert_{\mathcal{F}_\lambda} = e^{2\pi i\langle\beta,\lambda\rangle}
\end{equation*}
for $\beta\in\RR\alpha$ and $\lambda\in Q$, and we have $V_Q^{U(1)}=\mathcal{F}_0=\mathcal{H}$. It is not hard to show (see for example \cite[Theorem 2.3]{CKLR}) that the category of (direct sums of) Fock modules for $\mathcal{H}$ admits vertex tensor category structure, so that Corollary \ref{maintheorem} implies that the category of Fock modules with highest weights in $Q$ is a symmetric tensor category equivalent to $\mathrm{Rep}\,U(1)$. In particular, since $U(1)$ is abelian, we recover the well-known result that these Heisenberg Fock modules are simple currents, that is, the tensor product of two such irreducible modules is irreducible: $\mathcal{F}_\lambda\boxtimes\mathcal{F}_\mu\cong\mathcal{F}_{\lambda+\mu}$ for $\lambda,\mu\in Q$.

Alternatively, the fusion rules for Heisenberg Fock modules are well known to agree with those for finite-dimensional $U(1)$-modules, so that one could use Theorem \ref{fusionruletovrtxtens} to conclude the existence of vertex tensor category structure on the category of Fock modules with highest weights in $Q$. Then it follows from Corollary \ref{maintheorem} that the corresponding symmetric tensor category structure is equivalent to $\mathrm{Rep}\,U(1)$. 

One can further obtain vertex tensor category structure on the category of Fock modules with highest weights in $P$ by studying the analogous action of $U(1)\cong\RR\alpha/Q$ on $V_P$. In this case, the corresponding braided tensor category of $\mathcal{H}$-modules will be equivalent to the modification of $\mathrm{Rep}\,U(1)$ by a $3$-cocycle; we will describe this $3$-cocyle in detail in the next example. Note that here we are considering small semisimple subcategories of the full category $\mathcal{H}-\mathbf{mod}$ of grading-restricted, generalized $\mathcal{H}$-modules. In fact, the full category $\mathcal{H}-\mathbf{mod}$ admits vertex tensor category structure: the simple objects of $\mathcal{H}-\mathbf{mod}$ consist of all Fock modules with arbitrary highest weights, all simple modules are simple currents, and all objects are finite-length extensions of Fock modules.
\end{exam}

\begin{exam}
Now the action of $U(1)$ on $V_P$ as in the previous example extends to an action of $SU(2)$ on $V_P$, which descends to a faithful action of $SO(3)$ on $V_Q$; we will use these actions to obtain new vertex tensor categories. It is well known (see for instance \cite{DG}, \cite{Mil}), that the fixed-point subalgebra $V_Q^{SO(3)}=V_P^{SU(2)}$ is the simple Virasoro vertex operator algebra $L(1,0)$ with central charge $c=1$. The irreducible $L(1,0)$-modules appearing in the decomposition of $V_P$ as an $SU(2)$-module are the modules $L(1,\frac{n^2}{4})$ for $n\in\NN$, where $n^2/4$ denotes the lowest conformal weight (see \cite{Mil}). The functor $\Phi$ from $\rep SU(2)$ to $L(1,0)$-modules sends the $(n+1)$-dimensional irreducible $SU(2)$-module $V(n)$ to $L(1,\frac{n^2}{4})$.
 
 In \cite[Theorem 3.3]{Mil}, it was shown that the fusion rules among the $L(1,0)$-modules $L(1,\frac{n^2}{4})$ for $n\in\NN$ agree with the fusion rules for finite-dimensional irreducible $\mathfrak{sl}_2$-modules (equivalently, finite-dimensional irreducible continuous $SU(2)$-modules). Thus Theorem \ref{fusionruletovrtxtens} applies to show that the semisimple category $\cC_{V_P}$ of $L(1,0)$-modules generated by the $L(1,\frac{n^2}{4})$ for $n\in\NN$ has vertex tensor category structure. Then Corollary \ref{maintheorem} implies that the braided tensor category structure on $\cC_{V_P}$ is equivalent to the tensor category structure on $\rep SU(2)$ modified by the $3$-cocycle $(F,\Omega)$.\footnote{A more general result for $ADE$ weight lattice abelian intertwining algebras was stated in the dissertation \cite{St}, but the version of Theorem \ref{Jinjective} above that is needed there was not fully proved. This is why we only consider the case of $\mathfrak{sl}_2$ here, using \cite{Mil}.}
 
 Let us discuss the $3$-cocycle $(F,\Omega)$ defining the abelian intertwining algebra structure on $V_P$ in detail. First, the grading group $A$ is $P/Q\cong\ZZ/2\ZZ$. Then, careful examination of the construction in \cite[Chapter 12]{DL} reveals that we may take 
 \begin{equation*}
  F(\alpha_1+Q, \alpha_2+Q, \alpha_3+Q)=\left\lbrace\begin{array}{rcl}
       1 & \mathrm{if} & \alpha_1\in Q,\,\alpha_2\in Q,\,\mathrm{or}\,\,\alpha_3\in Q\\
       -1 & \mathrm{if} & \alpha_1,\alpha_2,\alpha_3\in\frac{\alpha}{2}+Q\\
                                                    \end{array}
\right. ,
 \end{equation*}
where $\alpha$ denotes the positive root of $\mathfrak{sl}_2$, and
\begin{equation*}
  \Omega(\alpha_1+Q, \alpha_2+Q) = \left\lbrace\begin{array}{rcl}
                                                1 & \mathrm{if} & \alpha_1\in Q\,\,\mathrm{or}\,\,\alpha_2\in Q\\
                                                -i & \mathrm{if} & \alpha_1,\alpha_2\in\frac{\alpha}{2}+Q
                                               \end{array}
\right.  .
 \end{equation*}
 Now, $V_Q$ contains the odd-dimensional irreducible representations of $SU(2)$, while $V_{\frac{\alpha}{2}+Q}$ contains the even-dimensional representations. Thus the braiding isomorphisms
\begin{equation*}
 \cR_{V(m), V(n)}: V(m)\otimes V(n)\rightarrow V(n)\otimes V(m)
\end{equation*}
for irreducible modules in $\rep_{P/Q, F, \Omega} SU(2)$ are given by:
\begin{equation*}
 \cR_{V(m),V(n)}(m_1\otimes m_2) =\left\lbrace\begin{array}{rcl}
                                             m_2\otimes m_1 & \mathrm{if} & m\in 2\ZZ\,\,\mathrm{or}\,\,n\in2\ZZ \\
                                             i\,(m_2\otimes m_1) & \mathrm{if} & m,n\in 2\ZZ+1\\
                                            \end{array}
\right. .
\end{equation*}
This formula allows us to calculate the $S$-matrix of $\cC_{V_P}$: 
\begin{equation*}
S_{m,n}=\mathrm{Tr}\,\cR_{L(1,\frac{n^2}{4}), L(1,\frac{m^2}{4})}\circ\cR_{L(1,\frac{m^2}{4}), L(1,\frac{n^2}{4})} = (-1)^{m n}(m+1)(n+1).
\end{equation*}
Since the $S$-matrix is an invariant of a ribbon tensor category structure, the fact that some $S$-matrix entries are negative shows that $\cC_{V_P}$ is \textit{not} quite equivalent to the unmodified symmetric ribbon tensor category $\rep SU(2)$. On the other hand, the smaller category $\mathcal{C}_{V_Q}$ containing the irreducible modules $L(1,\frac{n^2}{4})$ for $n\in2\ZZ$ is symmetric and tensor equivalent to $\mathrm{Rep}\,SO(3)$.
\end{exam}

\begin{exam}
 If a vertex operator algebra is $C_2$-cofinite and CFT-type (all conformal weights are non-negative and its $0$th conformal weight space is spanned by the vacuum), then its full category $V-\mathbf{mod}$ of grading-restricted generalized modules has vertex tensor category structure by \cite[Proposition 4.1 and Theorem 4.13]{H-cofin}. Moreover, if $G$ is a finite solvable automorphism group of a $C_2$-cofinite, CFT-type vertex operator algebra $V$, then the fixed-point subalgebra $V^G$ is also $C_2$-cofinite by \cite{Miy2}. Consequently, Corollary \ref{maintheorem} implies that in this situation, the irreducible $V^G$-modules appearing in the decomposition of $V$ as a $V^G$-module generate a symmetric tensor subcategory of $V^G-\mathbf{mod}$ equivalent to $\rep G$.
 
 As an explicit example in this setting, we may consider the symplectic fermion vertex operator superalgebra $SF(d)$ of $d$ pairs of symplectic fermions studied in \cite{Ab}, whose automorphism group is the symplectic group $Sp(2d,\CC)$. The even vertex operator subalgebra $SF^+(d)$ is $C_2$-cofinite and has automorphism group $Sp(2d,\CC)/\langle P\rangle$, where $P$ is the parity involution of $SF(d)$. Consequently, if $G\leq Sp(2d,\CC)$ is any finite solvable subgroup containing $P$, then $SF(d)^G=SF^+(d)^{G/\langle P\rangle}$ is $C_2$-cofinite, and thus $SF(d)^G-\mathbf{mod}$ has vertex tensor category structure. Then by Corollary \ref{maintheorem}, the irreducible $SF(d)^G$-modules appearing in the decomposition of $SF(d)$ as an $SF(d)^G$-module generate a braided tensor subcategory of $SF(d)^G-\mathbf{mod}$ equivalent to $\rep_{\ZZ/2\ZZ, 1,\Omega} G$ (recall from Example \ref{exam:VOSA} that here $\Omega(i_1+2\ZZ,i_2+2\ZZ)=(-1)^{i_1 i_2}$ for $i_1,i_2\in\ZZ$).
\end{exam}

As a further application of the tensor equivalence of Corollary \ref{maintheorem}, we will show that if $G$ is a finite automorphism group of a simple CFT-type vertex operator algebra $V$ and $V^G$ is strongly rational in the sense that it is simple, CFT-type, self-contragredient, $C_2$-cofinite, and rational, then $V$ is also strongly rational. This result was obtained previously in \cite[Lemma 4.2]{ADJR} under the strong additional assumption that all irreducible $V^G$-modules are non-negatively graded, with a non-zero conformal weight $0$ space occurring only in $V^G$ itself. Here, the role of Corollary \ref{maintheorem} is to show that the categorical dimension of $V$ in the modular tensor category $V^G-\mathbf{mod}$ is non-zero; this is needed for showing that $V-\mathbf{mod}$ inherits semisimplicity from $V^G-\mathbf{mod}$. Note that the converse question of whether strong rationality of $V$ implies that of $V^G$, resolved for solvable $G$ in \cite{CarM}, is certainly more difficult.
\begin{theo}\label{VregifVGis}
 Suppose $V$ is a simple CFT-type vertex operator algebra and $G$ is a finite automorphism group of $V$. If $V^G$ is strongly rational, then $V$ is also strongly rational.
\end{theo}
\begin{proof}
 We need to verify that $V$ is rational, $C_2$-cofinite, and self-contragredient. The $C_2$-cofiniteness of $V$ follows from \cite[Proposition 5.2]{ABD}: as a semisimple module for the $C_2$-cofinite vertex operator algebra $V^G$, $V$ is $C_2$-cofinite as $V^G$-module and hence also as $V$-module. That $V$ is self-contragredient follows from \cite[Corollary 3.2]{Li} since $V$ is simple and 
 \begin{equation*}
 L(1)V_{(1)} = \bigoplus_{\chi\in\widehat{G}} M_\chi\otimes L(1)(V_\chi)_{(1)} =L(1)V^G_{(1)}=0,
 \end{equation*}
where the last equality holds because $V^G$ is CFT-type and self-contragredient.
 
 For showing that $V$ is rational, note that $V$ is an object in the category $\mathcal{C}$ of (grading-restricted, generalized) $V^G$-modules, which by \cite{H-rigidity} is a semisimple modular tensor category. By \cite[Theorem 3.2 and Remark 3.3]{HKL}, $V$ is a commutative associative algebra in the modular tensor category $\mathcal{C}$, and by \cite[Theorem 3.4]{HKL}, the category of (grading-restricted, generalized) $V$-modules is the category $\mathrm{Rep}^0\,V$ of ``dyslectic'' modules for the algebra object $V$, using the notation of \cite{KO}. By Lemma 1.20, Theorem 3.2, and Theorem 3.3 of \cite{KO}, $\mathrm{Rep}^0\,V$ is semisimple provided that $V$ is simple and the categorical dimension $\mathrm{dim}_\mathcal{C} V\neq 0$. So we are reduced to proving $\dim_{\cC} V\neq 0$ and that semisimplicity of the grading-restricted, generalized module category of a CFT-type $C_2$-cofinite vertex operator algebra implies rationality. These will be proven in Propositions \ref{Vcatdim} and \ref{ssord-to-ssNgrad} below.
\end{proof}

\begin{propo}\label{Vcatdim}
 Assume that $V$ is a simple vertex operator algebra, $G$ is a finite automorphism group of $V$, and the irreducible modules $V_\chi$ for $\chi\in\widehat{G}$ occurring in the decomposition of $V$ as a $V^G$-module are objects in a vertex tensor category $\mathcal{C}$ of $V^G$-modules. Then $\mathrm{dim}_\mathcal{C} V=\vert G\vert$.
\end{propo}

\begin{proof}
 The decomposition $V=\bigoplus_{\chi\in\widehat{G}} M_\chi\otimes V_\chi$ implies that
 \begin{equation*}
  \mathrm{dim}_{\mathcal{C}} V =\sum_{\chi\in\widehat{G}} (\dim_{\mathbb{C}} M_\chi)(\dim_{\mathcal{C}} V_\chi).
 \end{equation*}
Since the twists on the ribbon categories $\mathrm{Rep}\,G$ and $\mathcal{C}_V$ are trivial (in the case of $\mathcal{C}_V$ this is because the twist $e^{2\pi i L(0)}$ equals the identity on the $\mathbb{Z}$-graded $V^G$-module $V$), $\Phi: \mathrm{Rep}\,G\rightarrow \mathcal{C}_V$ is an equivalence of ribbon categories. Thus
\begin{equation*}
 \mathrm{dim}_\mathcal{C} V_\chi = \mathrm{dim}_\mathcal{C} \Phi(M_\chi^*)=\dim_{\mathbb{C}}  M_\chi^*.
\end{equation*}
Consequently we calculate
\begin{equation*}
 \dim_\mathcal{C} V=\sum_{\chi\in\widehat{G}} (\dim_\mathbb{C} M_\chi)(\dim_{\mathbb{C}} M_\chi^*)=\sum_{\chi\in\widehat{G}} \dim_{\mathbb{C}} \mathrm{End}_\mathbb{C} M_\chi =\dim_\mathbb{C} \mathbb{C}[G]=\vert G\vert.
\end{equation*}

\end{proof}

\begin{propo}\label{ssord-to-ssNgrad}
 If the category of grading-restricted, generalized modules for a CFT-type, $C_2$-cofinite vertex operator algebra $V$ is semisimple, then $V$ is rational.
\end{propo}
\begin{proof}
 This is essentially the content of Lemma 3.6 and Proposition 3.7 of \cite{CarM}, but here we provide an alternate proof that uses a slightly different collection of results from the vertex operator algebra literature. We need to show that if $W=\bigoplus_{n\geq 0} W(n)$ is an $\mathbb{N}$-gradable weak $V$-module, then $W$ is the direct sum of (possibly infinitely many) irreducible grading-restricted submodules.
 
 First assume that $W$ is generated by a homogeneous vector $w\in W(N)$ for some $N\geq 0$. Then by \cite[Proposition 4.5.6]{LL},
 \begin{equation*}
  W=\mathrm{span}\lbrace v_n w\,\vert\,v\in V, n\in\mathbb{Z}\rbrace,
 \end{equation*}
and hence
\begin{equation*}
 W(N)=\mathrm{span}\lbrace v_{\mathrm{wt}\,v-1} w\,\vert\,v\,\mathrm{homogeneous}\rbrace.
\end{equation*}
This means that $W(N)$ is singly-generated as a module for the $N$th Zhu's algebra $A_N(V)$ as in \cite{DLM2}. Because $V$ is $C_2$-cofinite, $A_N(V)$ is finite dimensional by \cite[Corollary 5.5]{Bu}, and hence $W(N)$ is finite dimensional. Then \cite[Corollary 5.6]{Bu} shows that $W$ is an ordinary grading-restricted $V$-module. Since the category of grading-restricted $V$-modules is semisimple by assumption, it follows that $W$ is a direct sum of irreducible grading-restricted modules in this case.

Now take a general $\mathbb{N}$-gradable weak module $W$. Then $W=\sum_{n\geq 0}\sum_{w\in W(n)} V\cdot w$, where $V\cdot w$ is the $\mathbb{N}$-gradable weak $V$-submodule of $W$ generated by a homogeneous $w$. By the previous case, each $V\cdot w$ is a grading-restricted and thus semisimple $V$-module. This means that $W$ is a sum, and therefore also a direct sum, of irreducible grading-restricted $V$-modules, completing the proof that $V$ is rational.
\end{proof}

\begin{rema}
  Under the strong assumption that $V$ is $g$-rational for all $g\in G$, Proposition \ref{Vcatdim} was essentially proven in \cite[Theorem 6.3]{DJX}. The only difference is that \cite{DJX} is concerned with ``quantum dimensions'' of $V^G$-modules defined in terms of characters. When $V^G$ is strongly rational and every irreducible $V^G$-module other than $V^G$ itself has strictly positive conformal weights, the results of \cite{DJX} show that these quantum dimensions are actually Frobenius-Perron dimensions in the fusion category of $V^G$-modules. So in this setting we see that the categorical and Frobenius-Perron dimensions of modules in $\mathcal{C}_V$ coincide. Note, however, that Proposition \ref{Vcatdim} requires no rationality assumption on $V$.
\end{rema}

\begin{rema}
 The formula $\dim_\mathcal{C}\,V_\chi =\dim_\mathbb{C}\,M_\chi^*$ used in the proof of Proposition \ref{Vcatdim}, which holds for continuous actions of compact groups on $V$ as long as the vertex tensor category $\mathcal{C}$ exists, may be viewed as a partial resolution of Conjecture 6.7 in \cite{DJX}. However, the conjecture in \cite{DJX} is stated in terms of quantum dimensions which it is not clear necessarily exist when $V^G$ is not strongly rational. Of course, the categorical dimensions $\dim_\mathcal{C}\,V_\chi$ also do not exist unless we have braided tensor category structure on a suitable category $\mathcal{C}$ of $V^G$-modules.
\end{rema}

\end{document}